\documentclass[letterpaper,reqno]{amsart}
\usepackage{amsmath}
\usepackage{amsthm}
\usepackage{amssymb}
\usepackage[all]{xy}
\usepackage{bbold}
\usepackage{url}
\usepackage[mathscr]{eucal}

\numberwithin{equation}{section}

\newcommand{\eiso}{\simeq}

\newcommand{\closure}[1]{\overline{#1}}

\newcommand{\astt}{\bullet}

\newcommand{\arisorr}{\ar@{}[r]|*[@]{\simeq}}
\newcommand{\arisodd}{\ar@{}[d]|*[@]{\simeq}}
\newcommand{\arisoll}{\ar@{}[l]|*[@]{\simeq}}
\newcommand{\arisouu}{\ar@{}[u]|*[@]{\simeq}}
\newcommand{\arisor}{\ar@{}[r]|-*[@]{\simeq}}
\newcommand{\arisod}{\ar@{}[d]|-*[@]{\simeq}}
\newcommand{\arisol}{\ar@{}[l]|-*[@]{\simeq}}
\newcommand{\arisou}{\ar@{}[u]|-*[@]{\simeq}}

\newcommand{\tens}{\otimes}
\newcommand{\tensor}{\otimes}

\newcommand{\SH}{\mathrm{SH}}

\newcommand{\SHf}{\SH^\mathrm{fin}}
\newcommand{\SHpf}{\SHf_{(p)}}
\newcommand{\SHfp}{\SHpf}
\newcommand{\SHftor}{\SHf_{\mathrm{tor}}}

\newcommand{\Fp}{\mathbb{F}_p}

\newcommand{\eps}{\epsilon}
\renewcommand{\v}{v}
\newcommand{\p}{\mathfrak{p}}

\newcommand{\Z}{\mathbb{Z}}
\newcommand{\Q}{\mathbb{Q}}
\newcommand{\cZ}{\mathcal{Z}}
\newcommand{\cW}{\mathcal{W}}

\newcommand{\T}{\mathcal{T}}
\newcommand{\I}{\mathcal{I}}
\newcommand{\J}{\mathcal{J}}

\newcommand{\unit}{\mathbb{1}}
\newcommand{\K}{\mathcal{K}}
\renewcommand{\L}{\mathcal{L}}

\newcommand{\C}{\mathcal{C}}
\newcommand{\D}{\mathcal{D}}
\renewcommand{\P}{\mathcal{P}}
\newcommand{\cQ}{\mathcal{Q}}
\newcommand{\xra}{\xrightarrow}
\newcommand{\ra}{\rightarrow}

\newcommand{\Si}{\Sigma}

\newcommand{\ftri}[6]{#1 \xrightarrow{#4} #2 \xrightarrow{#5} #3 \xrightarrow{#6} \Si #1}

\newcommand{\ftrii}[6]{\xymatrix@1{ #1 \ar[r]^{#4} & #2 \ar[r]^{#5} & #3 \ar[r]^{#6} & \Si #1 }}

\newcommand{\lfraction}[5]{(\xymatrix@1{ #1 \ar[r]^-{#2} & #3 & #5 \ar[l]_-{#4}})}
\newcommand{\lfractionstretch}[6]{(\xymatrix@1 @C=#6{ #1 \ar[r]^-{#2} & #3 & #5 \ar[l]_-{#4}})}

\newcommand{\id}{\mathrm{id}}

\DeclareMathOperator{\colim}{colim}
\DeclareMathOperator{\Spech}{Spec^h}
\DeclareMathOperator{\Spec}{Spec}
\DeclareMathOperator{\Spc}{Spc}
\DeclareMathOperator{\End}{End}

\DeclareMathOperator{\supp}{supp}

\DeclareMathOperator{\cone}{cone}

\DeclareMathOperator{\Center}{Center}

\DeclareMathOperator{\op}{op}
\DeclareMathOperator{\rep}{rep}

\newtheorem{theorem}[equation]{Theorem}
\newtheorem{lemma}[equation]{Lemma}
\newtheorem{proposition}[equation]{Proposition}
\newtheorem{corollary}[equation]{Corollary}

\theoremstyle{definition}
\newtheorem{definition}[equation]{Definition}
\theoremstyle{remark}
\newtheorem{remark}[equation]{Remark}
\theoremstyle{remark}
\newtheorem{example}[equation]{Example}
\theoremstyle{definition}

\theoremstyle{definition}

\theoremstyle{definition}
\newtheorem{notation}[equation]{Notation}

\title[Higher Comparison Maps]{Higher Comparison Maps for the Spectrum of a Tensor Triangulated Category}
\author{Beren Sanders}
\date{\today}

\begin{document}
\begin{abstract}
	For each object in a tensor triangulated category, we construct a natural
	continuous map from the object's support---a closed subset of the
	category's triangular spectrum---to the Zariski spectrum of a certain
	commutative ring of endomorphisms.  When applied to the unit object this
	recovers a construction of P.~Balmer.  These maps provide an iterative
	approach for understanding the spectrum of a tensor triangulated category
	by starting with the comparison map for the unit object and iteratively
	analyzing the fibers of this map via ``higher'' comparison maps.  We
	illustrate this approach for the stable homotopy category of finite
	spectra.  In fact, the same underlying construction produces a whole
	collection of new comparison maps, including maps associated to (and
	defined on) each closed subset of the triangular spectrum.  These latter
	maps provide an alternative strategy for analyzing the spectrum by
	iteratively building a filtration of closed subsets by pulling back
	filtrations of affine schemes.
\end{abstract}

\maketitle
\tableofcontents

\section*{Introduction}
Many triangulated categories arising in nature come equipped with natural
\mbox{$\tensor$-product} structures---that is, they are \emph{tensor}
triangulated categories---and in recent years there has been a growing
appreciation for the significance of these $\tensor$-structures.  For example,
a (nice) scheme can be recovered from the tensor triangulated structure of its
derived category of perfect complexes, but not from the triangulated structure
alone (see \cite[Remark~64]{Balmer_SSS}, for example).  Using the
$\tensor$-structure, Paul Balmer \cite{Balmer_Spectrum} has introduced the
\emph{spectrum} of a tensor triangulated category.  Just as the spectrum of a
commutative ring provides a geometric approach to commutative algebra, the
spectrum of a tensor triangulated category provides a geometric approach to the
study of tensor triangulated categories---an approach referred to as
\emph{tensor triangular geometry} by its originators.  The present paper makes
a contribution to tensor triangular geometry and the antenatal reader is
referred to \cite{Balmer_TTG} for an introduction to this relatively new field
and for additional background that leads to the present work.

Determining the spectrum of a given tensor triangulated category is a highly
non-trivial problem, which is essentially equivalent to classifying the thick
triangulated $\tensor$-ideals in the category---in other words, classifying the
objects of the category up to the naturally available structure:
$\tensor$-products, $\oplus$-sums, $\oplus$-summands, suspensions, and
cofibers.  Major classification theorems in algebraic geometry, modular
representation theory and stable homotopy theory give complete descriptions of
the spectrum in several important examples, but one of the goals of tensor
triangular geometry is to go the other way---to develop techniques for
determining the spectrum (and thereby solve the classification problem), or to
at least say something interesting about the spectrum when a full determination
proves to be too ambitious.

In any tensor triangulated category, the endomorphism ring of the unit is
commutative, and the first step towards saying something about the spectrum of
a general tensor triangulated category was taken in \cite{Balmer_SSS} where
continuous maps 
\begin{equation*}
	\label{balmer_comparison_maps}
	\rho : \Spc(\K) \ra \Spec(\End_\K(\unit)) \qquad\text{ and }\qquad
\rho^\bullet: \Spc(\K) \ra \Spech(\End_\K^\bullet(\unit))
\end{equation*}
were defined going from the triangular spectrum to the (homogeneous) spectrum
of the (graded) endomorphism ring of the unit.  These ``comparison maps'' are
often surjective and so attention focusses on understanding conditions under
which they are injective and more generally on understanding their fibers.  If
$\K = D^\text{perf}(A)$ is the derived category of perfect complexes of a
commutative ring~$A$, then $\End_\K(\unit)$ is isomorphic to $A$ and $\rho$
turns out to be a homeomorphism.  This can be proved directly 
and provides an alternative proof of the affine case of the
Hopkins-Neeman-Thomason theorem.  On the other hand, if $G$ is a finite group,
$k$ is a field, and $\K = D^b(kG\text{-mod})$ with $\tens=\tens_k$, then
$\End_\K^\bullet(\unit)$ is group cohomology $H^\bullet(G,k)$ and it is known
using the classification theorem of Benson-Carlson-Rickard that the map
$\rho^\bullet$ is a homeomorphism.
A more direct proof of the injectivity of $\rho^\bullet$ in this example would
provide a new proof of the Benson-Carlson-Rickard theorem.  In general,
however, one cannot expect the (graded) endomorphisms of the unit to determine
the global structure of the whole category and we are left with the important
general problem of understanding the fibers of these comparison maps.

In this paper, we will construct new comparison maps which generalize those
mentioned above.  More specifically, for each object $X$ in a tensor
triangulated category $\K$ we will define maps
\begin{equation*}
\label{new_comparison_maps}
\rho_X :\supp(X)\ra \Spec(R_X) \qquad\text{ and }\qquad \rho_X^\astt : \supp(X)\ra \Spech(R_X^\astt)
\end{equation*}
from the support of $X$ (a closed subset of the triangular spectrum) to the
(homogeneous) spectrum of a certain (graded-)commutative ring of (graded)
endomorphisms of $X$, which recover the original comparison maps when
$X=\unit$.
The author's initial interest in these new comparison maps stems from the fact
that they provide a method for studying the fibers of the original maps.  This
in turn leads to an iterative strategy for studying the spectrum based on a
repeated analysis of the fibers of a sequence of generalized comparison maps.
The idea runs as follows. Given an arbitrary tensor triangulated category $\K$,
we can take the unit object and consider the comparison map
$\rho_\unit:\Spc(\K)\ra \Spec(R_\unit)$.  Understanding the fibers of this map
reduces by a localization technique to the case when $R_\unit$ is a local ring.
If the unique closed point $\mathfrak{m} = \langle f_1, \ldots, f_n \rangle$ is
finitely generated then it is straightforward to show that the fiber
$\rho_\unit^{-1}(\{\mathfrak{m}\})$ is equal to the support of the object $X_1
:= \cone(f_1) \tens \cdots \tens \cone(f_n)$.  This fiber can then be examined
more closely by considering the ``higher'' comparison map $\rho_{X_1} :
\supp(X_1)\rightarrow \Spec (R_{X_1})$ associated with the object $X_1$.  The
same procedure can then be used to study the fibers of $\rho_{X_1}$ and the
process repeats itself.  Following any particular thread in this process
produces a linear filtration
\begin{equation*}
	\Spc(\K) \supset \supp(X_1) \supset \supp(X_2) \supset \cdots \supset \supp(X_n)
\end{equation*}
which can be extended for however long the rings involved possess finitely
generated primes.

One of the difficulties with this method is that to understand the fiber over a
non-closed point we must first apply a localization procedure.  The reason is
that for a finitely generated prime $\mathfrak{p} = \langle f_1,\ldots,f_n
\rangle$, the support of $\cone(f_1)\tens \cdots \tens \cone(f_n)$ is actually
the preimage of the closure $\closure{\{\mathfrak{p}\}} = V(\mathfrak{p})$
rather than the fiber over~$\mathfrak{p}$.  More generally, $\rho_X^{-1}(V(I))
= \supp(\cone(f_1)\tens \cdots \tens \cone(f_n))$ for any finitely generated
ideal $I = \langle f_1, \ldots, f_n \rangle$.  Thus, rather than examining the
fibers of a comparison map $\rho_X$, an alternative strategy is to take a look
at the preimages of all of the Thomason closed subsets $V(I) \subset
\Spec(R_X)$.\footnote{Recall that a Thomason closed subset is the same thing as
	a closed subset whose complement is quasi-compact. In the case of an affine
	scheme $\Spec(A)$ this a closed set of the form $V(I)$ for a finitely
	generated ideal $I \subset A$, while in the case of $\Spc(\K)$ this is a
	closed subset of the form $\supp(a)$ for an object $a \in \K$. These
	notions will be reviewed in Section~\ref{section:spectral_and_thomason}.}
Choosing generators of the ideal $I$ provides us with an object of $\K$ whose
support is the closed subset $\rho_X^{-1}(V(I))$ and we can examine this subset
further via the comparison map associated with this ``generator''
object.\footnote{A ``generator'' of a closed subset $\cZ \subset \Spc(\K)$ is
	an object $a \in \K$ with $\supp(a) = \cZ$.}

Both of these strategies suffer from the fact that (a) they only deal with
finitely generated primes and Thomason closed subsets (which may be an
undesirable limitation in non-noetherian situations) and (b) they involve
non-canonical choices of generators.  The fundamental idea in both approaches
is to examine a Thomason closed subset $\cZ\subset \Spc(\K)$ by the comparison
map associated with an object which generates $\cZ$, but this comparison map
depends on the choice of generator.  Such considerations lead to the desire for
a ``generator-independent'' comparison map which only depends on the Thomason
closed subset on which it is defined, and more generally for a comparison map
associated to \emph{every} closed subset of the spectrum.

Indeed, the map $\rho_X$ is just one of a host of new comparison maps
introduced in this paper.  The most general construction associates a natural,
continuous map \[\rho_\Phi : \bigcap_{X \in \Phi} \supp(X) \ra \Spec(R_\Phi) \]
to each set of objects $\Phi \subset \K$ that is closed under the
$\tens$-product.  Taking $\Phi = \{X^{\tens n} \mid n \ge 1\}$ gives the map
$\rho_X$ above, while taking $\Phi = \{ a \in \K \mid \supp(a) \supset \cZ \}$
gives a map $\rho_\cZ: \cZ \ra \Spec(R_\cZ)$ associated to (and defined on) an
arbitrary closed subset of the spectrum.  Following on from the previous
discussion, the latter ``closed set'' comparison maps $\rho_\cZ$ afford perhaps
the most robust strategy for studying the spectrum.  The idea is to iteratively
build a filtration of closed subsets by pulling back filtrations of the affine
schemes $\Spec(R_\cZ)$.  This idea has the advantage that it utilizes all
closed subsets (not just Thomason ones) and is purely deterministic: no choices
are involved.  The hope is that ultimately the filtration will become fine
enough to completely determine the spectrum.  Although certain difficulties
prevent these strategies from working out in the full generality that one might
hope, the author nevertheless considers them to be the primary justification
for the theory developed in this paper.

In any case, one example where none of the difficulties arise is the stable
homotopy category of finite spectra $\SHf$.  This is an elusive example for
tensor triangular geometry.  Although the structure of the space $\Spc(\SHf)$
is known via the work of Devinatz, Hopkins and Smith (see
\cite{HS98} and \cite[Section~9]{Balmer_SSS}), the unit comparison map
\begin{equation*}
	\label{diagram:stable_rho_map}
	\xymatrix @R=1em{
 &		\C_{2,\infty} \ar@{-}[d] & \C_{3,\infty}\ar@{-}[d] & \cdots & \C_{p,\infty}\ar@{-}[d] & \cdots \\
 \Spc(\SHf) = \ar[dddddd]_{\rho_{\unit}}&		\vdots \ar@{-}[d]& \vdots \ar@{-}[d] && \vdots\ar@{-}[d] \\
	&	\C_{2,n+1}\ar@{-}[d] & \C_{3,n+1} \ar@{-}[d] & \cdots& \C_{p,n+1}\ar@{-}[d] &\cdots \\
	&	\C_{2,n}\ar@{-}[d] & \C_{3,n} \ar@{-}[d] & \cdots& \C_{p,n}\ar@{-}[d] & \cdots\\
&		\vdots\ar@{-}[d] & \vdots \ar@{-}[d] & & \vdots\ar@{-}[d] \\
&		\C_{2,2}\ar@{-}[drr] & \C_{3,2} \ar@{-}[dr] & \cdots& \C_{p,2} &\cdots \\
&		&& \SHftor \ar@{-}[ur] && \\
\Spec(\Z) = & 2\Z \ar@{-}[drr] & 3\Z \ar@{-}[dr]& \cdots & p\Z & \cdots \\
 &&& (0) \ar@{-}[ur]
}
\end{equation*}
is far from injective
and understanding the fibers (which are given by the Morava $K$-theories) is related
to the important problem of understanding
residue fields in tensor triangular geometry.
In any case, the
iterative 
procedure
we have mentioned above works out very nicely in this example,
and it provides one illustration of how the higher comparison maps can work out
in practice.
However, determining the comparison maps in this example---in
particular, determining the structure of the rings $R_X^\astt$---requires the
full strength of the results in \cite{HS98} on nilpotence and periodicity in
stable homotopy theory.  In particular, it presupposes knowledge of the
classification of thick subcategories in $\SHf$.  Nevertheless, these results
allow us to show that the new comparison maps refine the view of $\SHf$ provided
by Balmer's original comparison map $\rho_\unit$.  This is an important test for our theory
as other generalizations of the original maps have failed to provide additional
insight into this example.

The primary portion of this paper is devoted to the construction of the new
comparison maps and laying the foundations of their basic theory.  For example,
we establish their naturality, show that passing to the idempotent completion
does not change anything, and develop a technique for localizing with respect
to primes in $R_\Phi$ (which has been alluded to in the discussion above and
generalizes the ``central localization'' of \cite{Balmer_SSS}).  Other results
include establishing that the object comparison maps $\rho_X$ are invariant
under a number of natural operations that can be performed on the object $X$
such as taking suspensions, or duals, or $\tens$-powers, etc.  In addition, we
establish some connections of a topological nature between the target affine
scheme $\Spec(R_\Phi)$ and the domain of $\rho_\Phi$; for example, we show that
the domain is connected if and only if $\Spec(R_\Phi)$ is connected. Other
results of that nature include establishing that the image of $\rho_\Phi$ is
always dense in $\Spec(R_\Phi)$.  

There remain
a number of unresolved questions and speculations related to the comparison
maps defined in this paper and it will take more examples to determine the
value of these constructions. 
Finally, it is worth mentioning that
\cite{GregIvo_2rings} has also defined generalizations of the original
comparison maps from \cite{Balmer_SSS}. However, that work focuses on
invertible objects in the category and goes in quite a different direction than
the present paper.

\section{Tensor triangulated categories}
A tensor triangulated category is a triangulated category $(\K,\Si)$ together with a ``compatible'' 
symmetric monoidal structure $(\K,\tensor,\unit)$.
More precisely, it is required that 
$- \tensor a : \K \ra \K$ and $a\tensor -:\K \ra \K$ be 
exact functors for each $a \in \K$.
Implicit in this statement is the data of two natural isomorphisms
\begin{equation}
	\label{suspension_isomorphisms}
	\Si a \tensor b \simeq \Si(a \tensor b) \qquad\text{ and }\qquad a \tensor \Si b \simeq \Si(a\tensor b) 
\end{equation}
which relate the suspension and the tensor.
In addition to requiring $a \tensor -$ and $- \tensor a$ be exact functors, these
suspension isomorphisms are required to relate suitably with the
symmetry, associator, and unitor isomorphisms of the symmetric monoidal
structure (see \cite[Appendix~A.2]{HPS97} for details).
Finally, it is assumed that the diagram
\begin{eqnarray}
	\label{diagram:anticommutes}
	\begin{gathered}
	\xymatrix @=1.5em{
		\label{commute_or_anticommute}
		\Si a \tensor \Si b \arisod \arisorr & \Si(a \tensor \Si b) \arisod \\
		\Si(\Si a \tensor b) \arisorr & \Si^2 (a \tensor b) }
\end{gathered}
\end{eqnarray}
anti-commutes.

By a morphism of tensor triangulated categories we mean a functor $F:\C \ra \D$
that is both an exact functor of triangulated categories as well
as a strong $\tens$-functor. Moreover, 
some compatibility is required between
the various isomorphisms that are attached to the functor
($F \Si a \simeq \Si Fa$, $Fa \tens Fb \simeq F(a \tens b)$, $F\unit_\C \simeq \unit_\D$),
as well as between those isomorphisms and the suspension isomorphisms 
of the categories $\C$ and $\D$.
The compatibility axioms are mostly obvious and are not usually spelled out. 
One point to be made is that the naturality of our graded comparison maps
depends on the compatibility axiom which asserts that 
\begin{eqnarray}
	\label{diagram:morphism_compatibility}
	\begin{gathered}
\xymatrix @=1.5em{
		\Si(Fa \tens Fb) \arisorr \arisodd & \Si Fa \tens Fb \arisorr & F\Si a \tens Fb \arisodd \\
		\Si F(a \tens b) \arisorr & F \Si(a \tens b) \arisorr & F(\Si a \tens b)
	}
\end{gathered}
\end{eqnarray}
commutes.

	Many authors include stronger assumptions about the monoidal structure in
	their definition of a tensor triangulated category. For example, it is
	common to assume that the monoidal structure is a \emph{closed} symmetric monoidal structure
(i.e., that internal homs exist).
	For the general constructions of this paper, 
	we don't need anything more than a symmetric monoidal
	structure. 
	However, for some results we will need to assume that our tensor triangulated category
	is \emph{rigid}; in other words, 
	that every object is dualizable and that the ``taking duals'' functor
	$D: \K^{\op} \ra \K$ preserves exact triangles. See \cite[Definition~1.5]{Balmer_SSS} for a precise definition.
	Many tensor triangulated categories of interest are rigid, but
	we will be explicit about when this
	assumption is required.

\begin{notation}
For a collection of objects $\mathcal{E}$ in a tensor triangulated category~$\K$, 
$\langle \mathcal{E} \rangle$ will denote the thick $\tens$-ideal generated by $\mathcal{E}$.
\end{notation}

The reader is assumed to be familiar with the basic definitions and theory of
the spectrum $\Spc(\K)$ of an essentially small tensor triangulated category
$\K$ (introduced in \cite{Balmer_Spectrum}), as well as with the notion of a
\emph{local} tensor triangulated category (introduced in \cite[Section~4]{Balmer_SSS}).
In this paper, we will not explicitly state
the assumption
that $\K$ is essentially small, but it will be tacitly assumed any time we speak of~$\Spc(\K)$.
Although the spectrum has the structure of a locally ringed space,
for our purposes only its topological structure is relevant.
Bear in mind that the Balmer topology on $\Spc(\K)$ is \emph{not} the Zariski topology one would obtain
by mimicing the definition of the Zariski topology on the prime spectrum of a
commutative ring; from the point of view of spectral spaces (see below) it is the Hochster-dual of the Zariski topology.
The result is that 
some things in tensor triangular geometry behave a bit differently
than one might expect.
For example, closure in the Balmer topology goes \emph{down} rather than \emph{up}:
$\overline{ \{\P\} } = \{ \cQ \in \Spc(\K) \mid \cQ \subset \P \}$.
In particular, the closed points in $\Spc(\K)$ are the \emph{minimal} primes.
Another consequence of the differences between the Balmer and Zariski topologies is that
our comparison maps will be inclusion-\emph{reversing}.

\section{Spectral spaces and Thomason subsets}\label{section:spectral_and_thomason}
\begin{definition}
	A topological space is \emph{spectral} if it is $T_0$, quasi-compact, the quasi-compact open subsets are closed under finite intersection and form an open basis, and every non-empty irreducible closed
	subset has a generic point.
\end{definition}
Hochster \cite{Hochster69} showed that a topological space is spectral if and only if it is homeomorphic to the Zariski spectrum of a commutative ring.
On the other hand, the spectrum of a tensor triangulated category is spectral and it follows from the results of Hochster, Thomason, and Balmer that every spectral space arises in this way.
\begin{definition}
	A subset $\mathcal{Y} \subset \mathcal{X}$ of a spectral space is \emph{Thomason} if it is a
	union of closed subsets each of which has quasi-compact complement.
\end{definition}
Hochster showed that every spectral space admits a ``dual'' spectral topology whose open sets are precisely the Thomason subsets.
The nomenclature comes from the prominent role these ``dual-open'' sets play in the work of Thomason \cite{Thomason97}.
In this paper, we will be interested in subsets of spectral spaces that are both Thomason and closed:
\begin{lemma}
	\label{lemma:thomason_subsets}
	Let $\K$ be a tensor triangulated category and let $\mathcal{Z}$ be a closed subset of $\Spc(\K)$.
	The following are equivalent:
	\begin{enumerate}
		\item $\mathcal{Z}$ is Thomason;
		\item $\mathcal{Z}$ has quasi-compact complement;
		\item $\mathcal{Z} = \supp(a)$ for some $a \in \K$.
	\end{enumerate}
\end{lemma}
\begin{proof}
	We will sketch the proof of (1) implies (2)
	since (2) clearly implies (1) and \cite[Proposition~2.14]{Balmer_Spectrum} gives the equivalence of (2) and (3).
	For any closed subset $\cZ \subset \Spc(\K)$, $\K_\cZ := \{ a \in \K \mid \supp(a) \subset \cZ\}$ is a thick $\tens$-ideal
	and it is easily checked from the definitions that $\Spc(\K)\setminus \cZ \subset \{ \P \in \Spc(\K) \mid \P \supset \K_\cZ\}=:V(\K_\cZ)$.
	On the other hand, if $\cZ$ is Thomason then one readily checks that the reverse inclusion holds using the equivalence of (2) and (3).
	Thus, if $\cZ$ is Thomason and closed then $\Spc(\K)\setminus \cZ = V(\K_\cZ) \simeq \Spc(\K/\K_\cZ)$ and the spectrum of any
	tensor triangulated category is quasi-compact
	(by \cite[Corollary~2.15]{Balmer_Spectrum}).
\end{proof}
\begin{definition}
	A \emph{spectral map} between spectral spaces is a continuous map with the property that the preimage of any
	quasi-compact open subset is again quasi-compact.
	This is equivalent to being a continuous map that is also continuous with respect to the dual spectral topologies (although one direction of this equivalence is not immediate).
\end{definition}
\begin{remark}
	Any closed subspace of a spectral space is also spectral, as is
	the homogeneous spectrum of a graded-commutative graded ring; for the latter, see \cite[Proposition~2.5]{BKS07} and
	\cite[Proposition~2.43]{GregIvo_2rings}.
	Our comparison maps will be spectral maps defined on closed subsets of the spectrum of a tensor triangulated category
	and mapping to the (homogeneous) spectrum of a (graded-)commutative (graded) ring.
\end{remark}
\begin{remark}
	It is well-known 
	that the Thomason closed subsets of an affine scheme $\Spec(A)$ are those
	closed sets of the form $V(I)$ for a finitely generated ideal $I\subset A$, and 
	similarly for the homogeneous spectrum of a
	(graded-)commutative graded ring; cf.~Lemma~\ref{lemma:thomason_subsets}, $(1) \Leftrightarrow (2)$, above, and 
	\cite[Lemma~2.2]{BKS07}.
\end{remark}


\section{Basic constructions}
\label{section:basic_construction}
It is now time to introduce the new comparison maps. As mentioned in the
introduction, there are actually several different constructions, but they are
closely related and the fundamental ideas are exposed in the simplest example.
In all cases, there are graded and ungraded versions.  The proofs for the
graded constructions are essentially the same as for the ungraded ones, but the
ideas are more transparent in the ungraded setting.
The notion of a ``tensor-balanced'' endomorphism will play a central role
in these constructions.

\begin{definition}
	An endomorphism $f:X \ra X$ in a tensor triangulated category is said to be 
	\emph{$\tens$-balanced}
	if $f \tens X = X \tens f$ as an endomorphism of $X \tens X$.
\end{definition}
\begin{remark}
	The following lemma was established in \cite[Proposition~2.13]{Balmer_SSS}
	in the case when $f:\unit \ra \unit$ is an arbitrary endomorphism of the
	unit and was a crucial technical result used in the construction of the
	original unit comparison maps.  The key to generalizing the result to
	endomorphisms of an arbitrary object~$X$ is to restrict ourselves to
	$\tens$-balanced endomorphisms.
\end{remark}
\begin{lemma}
	\label{lemma:killscone}
	If $f:X \ra X$ is a $\tens$-balanced endomorphism then $f^{\tens 2}\tens \cone(f)=0$.
\end{lemma}
\begin{proof}
Start with an exact triangle
$\ftri{X}{X}{\cone(f)}{f}{g}{h}$
	and observe that in the following morphism of exact triangles
\[ 
\xymatrix @C=5em{
	X\tens X \ar[r]^-{X\tens f} \ar[d]_-{f\tens X} & X\tens X \ar[r]^-{X\tens g} \ar[d]_-{f\tens X} \ar@{.>}[rd]^-0 & X\tens \cone(f) \ar[r]^-{X\tens h} \ar[d]^-{f\tens \cone(f)} \ar@{.>}[rd]^-0 &\Si(X\tens X) \ar[d]^-{\Si(f\tens X)}\\
	X\tens X \ar[r]^-{X\tens f} & X\tens X \ar[r]^-{X\tens g} & X\tens \cone(f) \ar[r]^-{X\tens h} &\Si(X\tens X)\\
}
\]
the middle diagonal 
is zero because $(X\tens g)\circ(f\tens X) = (X\tens g)\circ(X \tens f) = X\tens(g\circ f) = 0$
and the rightmost diagonal is zero for similar reasons.
This implies that the map $f \tens \cone(f)$ factors through $X \tens g$ and $X \tens h$ and hence $(f \tens \cone(f))^2 = 0$.
It follows that $f^{\tens 2} \tens \cone(f) = 0$ 
by observing that $f^{\tens 2} = (X \tens f)\circ (f \tens X) = X \tens f^2$.
\end{proof}
\begin{notation}\label{notation:E_X}
	Let $E_X := \{f \in [X,X] \mid f \tens X = X \tens f\}$ denote the collection of $\tens$-balanced endomorphisms of $X$.
\end{notation}
\begin{proposition}
	\label{proposition:local_ring}
	For each object $X$ in a tensor triangulated category $\K$, $E_X$ is
	an inverse-closed
	subring of the endomorphism ring $[X,X]$.
	If $(0)$ is a prime in $\K$, for example if $\K$ is rigid and local, then $E_{X}$ is a 
	local ring provided that $X\neq 0$.
\end{proposition}
\begin{proof}
	The first statement follows easily from the definitions.
On the other hand, suppose that the zero ideal $(0)$ is a prime in $\K$ and that $X \neq 0$.
To prove that the non-zero ring $E_X$ is local
it suffices to show that the sum of two non-units is again a non-unit.
To this end, let $f_1,f_2 \in E_X$
and suppose that $f_1+f_2$ is a unit.
By Lemma~\ref{lemma:killscone}, $f_1^{\tens 2} \tens \cone(f_1) = 0$
and $f_2^{\tens 2} \tens \cone(f_2) = 0$.
It follows that $(f_1+f_2)^{\tens n} \tens \cone(f_1) \tens \cone(f_2) = 0$ for $n \ge 3$ by
expanding $(f_1 + f_2)^{\tens n}$ using bilinearity of the $\tens$-product and applying the symmetry.
But the unit $f_1 + f_2$ is a categorical isomorphism and hence any 
$\tens$-power $(f_1 + f_2)^{\tens n}$ is also an isomorphism.
It follows that
$X^{\tens n} \tens \cone(f_1) \tens \cone(f_2) = 0$ for $n \ge 3$ and hence
that $\cone(f_1) = 0$ or $\cone(f_2) =0$
since $(0)$ is prime and $X \neq 0$ by assumption.
In other words, $f_1$ or
$f_2$ is an isomorphism (and hence a unit in~$E_X$).
\end{proof}
\begin{lemma}
	\label{lemma:Ex_induced}
	If $F : \K \ra \L$ is a morphism of tensor triangulated
	categories then the induced ring homomorphism 
	$[X,X]_\K \ra [FX,FX]_\L$
	restricts to a
	ring
	homomorphism $E_{\K,X} \ra E_{\L,FX}$.
\end{lemma}
\begin{proof}
	This follows from the fact that $F: \K \ra \L$ is a strong $\tens$-functor.
\end{proof}
\begin{remark}
These results reveal the crucial properties that are secured by restricting
ourselves to
$\tens$-balanced endomorphisms: they provide us 
with rings of endomorphisms that are local when the category is local, behave well with respect to tensor triangular functors,
and have the property that the units are the elements that are categorical isomorphisms. 
However, these rings are not necessarily commutative.
\end{remark}
\begin{theorem}
	\label{theorem:unnatural_ungraded_comparison_map}
	Let $\K$ be a tensor triangulated category and let $X$ be an object in~$\K$.
	For any commutative ring $A$ and ring homomorphism $\alpha:A \ra E_X$ 
	there is an inclusion-reversing,
	spectral map \[\rho_{X,A} : \supp(X) \ra \Spec(A) \] defined by 
	$\rho_{X,A}(\P) := \{ a \in A \mid \cone(\alpha(a)) \notin \P \}$.
\end{theorem}
\begin{proof}
	The localization $q:\K \ra \K/\P$ induces a ring 
	homomorphism $E_{\K,X} \ra E_{\K/\P,q(X)}$ and since $X \notin \P$ the
	target ring $E_{\K/\P,q(X)}$ is a local ring.
	For any element $f\in E_{\K,X}$ observe that $\cone(f) \notin \P$ iff $q(f)$ is not an isomorphism in $\K/\P$
	iff $q(f)$ is a non-unit in the local ring $E_{\K/\P,q(X)}$.
	Since the non-units in a local ring form a two-sided ideal, it follows that
	the preimage $\{f \in E_{\K,X} \mid \cone(f) \notin \P \}$ is a two-sided ideal of $E_{\K,X}$.
	Moreover, this ideal is ``prime'' in the sense that $\cone(f\cdot g) \notin \P$ implies that $\cone(f) \notin \P$ or $\cone(g)\notin \P$ 
	(by an application of the octahedral axiom).
	In any case, this ``prime'' ideal of the non-commutative ring $E_{\K,X}$ pulls back via $\alpha$ to a 
	genuine prime ideal $\rho_{X,A}(\P)$ of the commutative ring~$A$.
This establishes that the map $\rho_{X,A}$ is well-defined and it is clear from the definition that it is inclusion-reversing.

An arbitrary closed set for the
Zariski topology on $\Spec(A)$ is of the form
$V(E) = \{\mathfrak{p} \in \Spec(A) \mid \mathfrak{p} \supset E\}$ for
some subset $E \subset A$. One  
readily checks that
$\rho_{X,A}^{-1}(V(E)) = \bigcap_{a \in E} \supp(\cone(\alpha(a)))$ and we conclude that $\rho_{X,A}$ is continuous.
Moreover, if $V(E)$ has quasi-compact complement then $V(E) = V(a_1,\ldots,a_n)$ for some
finite collection $a_1,\ldots,a_n \in A$ and the 
preimage 
$\rho_{X,A}^{-1}(V(a_1,\ldots,a_n)) = \bigcap_{i=1}^n \supp(\cone(\alpha(a_i))) = 
\supp(\cone(\alpha(a_1))\tens \cdots \tens \cone(\alpha(a_n)))$
also has quasi-compact complement
by Lemma~\ref{lemma:thomason_subsets}.
\end{proof}
 All the results above have corresponding graded analogues.
For a graded endomorphism $f: \Si^k X \ra X$ we abuse notation and write $f \tens X=X\tens f$ when we really mean that 
the following diagram commutes:
\begin{equation}
	\label{diagram:graded_weakly_central}
	\begin{gathered}
	\xymatrix @=0.2in{
		\Si^k(X\tens X) \arisod \arisor & \Si^k X \tens X \ar[d]^{f\tens X} \\
		X\tens \Si^kX \ar[r]_{X\tens f} & X \tens X
}
\end{gathered}
\end{equation}
\begin{proposition}
	\label{proposition:graded_ring_thing}
	A graded subring $E_X^\bullet$ of the graded endomorphism ring $[X,X]_\astt$
	is defined by setting $E_X^i := \{ f \in [X,X]_i \mid f\tens X = X \tens f\}$.
	It has the property that a homogeneous element is a unit in $E_X^\bullet$ iff
	it is a unit in $[X,X]_\astt$ iff it is a categorical isomorphism.
	Moreover, if $(0)$ is a prime in $\K$, for example if $\K$ is rigid and local, then $E_X^\astt$ is gr-local provided that $X \neq 0$.
\end{proposition}
\begin{proof}
	The proof is similar to the ungraded version,
	one just needs to take the relevant suspension
	isomorphisms into account.
	In particular, 
the result of Lemma~\ref{lemma:killscone} holds for a graded endomorphism $f:\Si^k X \ra X$ satisfying $f \tens X = X \tens f$.
 On the other hand, one could save time and conclude that $E_X^\bullet$ is \mbox{gr-local} simply by invoking the fact that
 a $\Z$-graded ring $E_X^\astt$ is gr-local iff $E_X^0 = E_X$ is local 
 (see \cite[Theorem~2.5]{Li_monoid_graded_local_rings}).
\end{proof}
\begin{lemma}
	\label{lemma:graded_induced}
	If $F:\K \ra \L$ is a morphism of tensor triangulated
	categories then the induced graded ring homomorphism
	$[X,X]_{\K,\astt} \ra [FX,FX]_{\L,\astt}$ restricts to a
	graded ring homomorphism $E_{\K,X}^\astt \ra E_{\L,FX}^\astt$.
\end{lemma}
\begin{proof}
This involves verifying that a diagram commutes using the monoidal nature of the functor.
	The subtle point is that because of the suspension isomorphisms involved in
	\eqref{diagram:graded_weakly_central}
	one must utilize the compatibility axiom 
	\eqref{diagram:morphism_compatibility} for morphisms of tensor triangulated categories.
\end{proof}
\begin{theorem}
	\label{theorem:graded_comparison_map}
	Let $\K$ be a tensor triangulated category and let $X$ be an object in~$\K$.
	For any (graded-)commutative graded ring $A^\astt$ and
	graded ring homomorphism $\alpha:A^\astt \ra E_X^\astt$ there is an
	inclusion-reversing, spectral map
	\[ \rho_{X,A^\astt}^\astt : \supp(X) \ra \Spech(A^\astt) \]
	defined by $\rho_{X,A^\astt}^\astt(\P) := \{ a \in A^i \mid \cone(\alpha(a)) \notin \P \}_{i \in \Z}$.
\end{theorem}
\begin{proof}
The localization functor $q:\K \ra \K/\P$ induces a graded ring homomorphism
$E_{\K,X}^\astt \ra E_{\K/\P,q(X)}^\astt$ and since $X \notin \P$ the target
ring $E_{\K/\P,q(X)}^\astt$ is gr-local.
The homogeneous non-units in $E_{\K/\P,q(X)}^\astt$ therefore form a two-sided ideal; moreover, an application of the octahedral axiom shows that this
ideal is prime in the sense that the product of two homogeneous units is again a unit.
The pullback of this ideal to $A^\astt$ is a genuine homogeneous prime ideal of $A^\astt$ and this is exactly what
$\rho_{X,A^\astt}^\astt(\P)$ is defined to be. This shows that the map $\rho_{X,A^\astt}^\astt$ is well-defined and it is clear from
the definition that it is inclusion-reversing.
Showing that it is spectral involves an argument similar to the one given in the proof of Theorem~\ref{theorem:unnatural_ungraded_comparison_map}.
\end{proof}
\begin{remark}
	It is clear from the definitions that there is a commutative diagram
	\[\xymatrix @C=3em{
			\supp(X) \ar[dr]_{\rho_{X,A^0}} \ar[r]^-{\rho_{X,A^\astt}^\astt} & \Spech(A^\astt) \ar[d]^{(-)^0} \\
			& \Spec(A^0)
		}
	\]
	where $(-)^0$ is the surjective spectral map sending $\mathfrak{p} \in \Spech(A^\astt)$
	to $\mathfrak{p}^0 = \mathfrak{p} \cap A^0$. The surjectivity of this map is explained in \cite[Remark~5.5]{Balmer_SSS}.
\end{remark}
\begin{example}
	Any (graded-)commutative graded subring of $E_X^\astt$ yields an associated comparison map.
	Obvious examples include 
	the graded-center of $E_X^\astt$ and 
	$\{ f \in \Center [X,X]_\astt \mid f \tens X = X \tens f\}$.
	A more exotic example is given by
\begin{equation}
	\label{eq:exotic}
	\{f \in \Center [X,X]_\astt \mid f\tens X \in \Center [X^{\tens 2},X^{\tens 2}]_\astt\}.
\end{equation}
For this third example, note that if $f \tens X \in \Center[X^{\tens 2}, X^{\tens 2}]_\astt$
then it follows from the fact that the symmetry $\tau : X\tens X \xra{\sim} X \tens X$ is in
$[X^{\tens 2},X^{\tens 2}]$
that $f \in E_X^\astt$;
so~\eqref{eq:exotic} does indeed give 
a graded subring of $E_X^\astt$.
\end{example}
\begin{example}
	\label{example:exampledef}
If $X = \unit$ then the condition $f \tens X = X \tens f$ holds for any graded
endomorphism 
and it is well-known that the ring $E_\unit^\astt = [\unit,\unit]_\astt$ is graded-commutative.
The map $\rho_{\unit,[\unit,\unit]_\astt}^\astt$ is the original graded comparison map from \cite{Balmer_SSS}.
\end{example}
\begin{example}
	\label{example:graded_center}
Recall the notion of the graded-center $Z^\astt(\mathcal{T})$ of a
triangulated category $\mathcal{T}$ (see \cite{KrauseYe_center}, for example).
For a \emph{tensor} triangulated category $\mathcal{T}$
one can define a graded-commutative graded subring
of $Z^\astt(\mathcal{T})$
by setting
\[Z_\tens^i(\mathcal{T}) := \{ \alpha \in Z^i(\mathcal{T}) \mid X \tens \alpha_Y = \alpha_{X \tens Y} \text{ for every } X,Y\in \mathcal{T}\}\]
for each $i\in \Z$.
Observe that any $\alpha \in Z_\tens^i(\mathcal{T})$ is completely determined by $\alpha_\unit$ and there is an obvious isomorphism $Z_\tens^\astt(\mathcal{T}) \xra{\sim} [\unit,\unit]_\astt$.
However, the definition makes sense for any thick $\tens$-ideal $\I \subseteq \T$ and
$Z_\tens^\astt(\I)$ is not obviously so trivial for $\I \subsetneq \T$.
For any object $X \in \I$ there is a 
graded ring homomorphism
$Z_\tens^\astt(\I) \ra E_X^\astt$ given by $\alpha \mapsto \alpha_X$ and so we obtain a map
$\supp(X) \ra \Spech(Z_\tens^\astt(\I))$.
Explicitly, it maps a prime $\P$ to
$\{ \alpha \in Z_{\tens}^\astt(\I) \mid \cone(\alpha_X) \notin \P \}$.
However, for a fixed prime $\P$ and a fixed $\alpha \in Z_{\tens}^\astt(\I)$ the set
$\{ Y \in \I \mid \cone(\alpha_Y) \in \P\}$ is readily checked to be a thick $\tens$-ideal.
It follows that if $X$ generates $\I$ as a thick $\tens$-ideal then
$\{ \alpha \in Z_{\tens}^\astt(\I) \mid \cone(\alpha_X) \notin \P \}$
is the same as $\{ \alpha \in Z_{\tens}^\astt(\I) \mid \cone(\alpha_Y) \notin \P \text{ for some } Y \in \I\}$.
In other words, if $\I$ is generated as a thick $\tens$-ideal by a single object (equivalently, by a finite number of objects)
then every generator gives the exact same comparison map.
In conclusion, every finitely generated thick $\tens$-ideal $\I$ has an associated (generator-independent) comparison map
$\supp(\I) \ra \Spech(Z_\tens^\astt(\I))$
which sends a prime $\P$ to
$\{ \alpha \in Z_\tens^\astt(\I) \mid \cone(\alpha_Y) \notin \P \text{ for some } Y \in \I \}$. 
\end{example}
\begin{remark}
	In the proof of Theorem~\ref{theorem:unnatural_ungraded_comparison_map}, we
	saw how to associate a ``prime'' ideal of the non-commutative ring $E_X$ to
	any prime $\P \in \supp(X)$ which was then pulled back to a genuine prime
	ideal of a commutative ring $A$ via a map $A \ra E_X$.  A suitable theory
	of spectra for non-commutative rings might allow us to work directly with
	the ring $E_X$ but this avenue has not been pursued.  In any case, taking
	commutative rings mapping into $E_X$ is a flexible approach which provides
	for some interesting examples not obviously tied to the ring $E_X$ (e.g.,
	Example~\ref{example:graded_center} above).  On the other hand, although
	the maps $\rho_{X,A}$ are useful for some purposes, they will not typically
	be natural with respect to tensor triangular functors.  The problem is that
	although the construction of the ring $E_X$ is functorial (recall
	Lemma~\ref{lemma:Ex_induced}), the construction of various commutative
	rings $A$ mapping into $E_X$ will typically not be.  For example, the
	center of $E_X$ is not a functorial construction, nor is the graded-center
	of a triangulated category.  In the next section, we will replace $E_X$ by
	a functorial commutative ring $R_X$ and obtain a comparison map $\rho_X :
	\supp(X) \ra \Spec(R_X)$ which \emph{is} natural with respect to tensor
	triangular functors.  In fact, the construction of $\rho_X$ will be a
	special case of a much more general construction, which will provide us
	with additional examples of natural comparison maps.
\end{remark}

\section{Natural constructions}
Let $\Phi$ be a non-empty set of objects in a tensor triangulated category $\K$ that is closed under the
$\tens$-product ($a,b\in \Phi \Rightarrow a\tens b\in \Phi$).
For any object $X \in \K$ recall that $E_X$ denotes the ring of $\tens$-balanced endomorphisms of $X$.
\begin{lemma}
	\label{lemma:twisting_lemma}
	Suppose $f:X\ra X$ and $g:Y \ra Y$ are two endomorphisms with $g \in E_Y$.
	If $X \tens f \tens g = f \tens X \tens g$ then $f \tens g \in E_{X\tens Y}$.
\end{lemma}
\begin{proof}
	The commutativity of the diagram
\begin{equation*}\label{eq:symmetry_tricks}
	\xymatrix @C=5em{
			X \tens Y \tens X \tens Y \ar@/_4pc/[dd]_{\id} \ar[d]^{\tau \tens X \tens Y} \ar[r]^{X \tens Y \tens f \tens g} & X \tens Y \tens X \tens Y \ar[d]_{\tau \tens X \tens Y}\ar@/^4pc/[dd]^{\id}\\
			Y \tens X \tens X \tens Y \ar[d]^{\tau \tens X \tens Y} \ar@<0.25em>[r]^{Y\tens X \tens f \tens g} \ar@<-0.25em>[r]_{Y\tens f \tens X \tens g} & Y \tens X \tens X \tens Y \ar[d]_{\tau \tens X \tens Y}\\
			X \tens Y \tens X \tens Y \ar@/_4pc/[dd]_{\id} \ar[d]^{X \tens \tau \tens Y} \ar[r]^{f\tens Y \tens X \tens g} & X \tens Y \tens X \tens Y \ar[d]_{X \tens \tau \tens Y}\ar@/^4pc/[dd]^{\id}\\
			X \tens X \tens Y \tens Y \ar[d]^{X \tens \tau \tens Y} \ar@<0.25em>[r]^{f\tens X \tens Y \tens g}\ar@<-0.25em>[r]_{f \tens X \tens g \tens Y} & X \tens X \tens Y \tens Y \ar[d]_{X \tens \tau \tens Y}\\
			X \tens Y \tens X \tens Y \ar[r]^{f\tens g \tens X \tens Y}& X \tens Y \tens X \tens Y.
		}
	\end{equation*}
	verifies that $f \tens g \in E_{X\tens Y}$. 
\end{proof}
\begin{corollary}
	\label{corollary:tensor_Ex}
	For any pair of objects $X,Y\in\K$ the functors $-\tens Y$ and $X \tens -$ induce
	ring homomorphisms $E_X \ra E_{X \tens Y}$ and $E_Y \ra E_{X \tens Y}$.
\end{corollary}
\begin{definition}
	\label{definition:general_ring}
Define $R_\Phi$ to be the set $\{(X,f) : X \in \Phi, f\in E_X\}/\mathord{\sim}$ where
	$\sim$ is the
	smallest equivalence relation such that
	$(X,f) \sim (a\tens X, a \tens f)$ and $(X,f) \sim (X \tens a,f\tens a)$ for every $a \in \Phi$.
\end{definition}
\begin{notation}
A subscript $f_X$ will indicate that $f$ is an endomorphism of $X$ and 
$[f_X]$ will denote the image of $(X,f)$ in $R_\Phi$.
\end{notation}
\begin{lemma}
	\label{lemma:iso_forced}
	For any isomorphism $\alpha : X \xra{\sim} Y$ in $\K$,
	the isomorphism of rings $[X,X] \xra{\sim} [Y,Y]$ given by
	$f \mapsto \alpha \circ f \circ \alpha^{-1}$ restricts to give an
	isomorphism \mbox{$\alpha_* : E_X \xra{\sim} E_Y$}.
	If $f\in E_X$ then $f \tens Y = X \tens \alpha_\ast(f)$ as an endomorphism of~$X \tens Y$.
\end{lemma}
\begin{proof}
	This is routine from the definitions.
\end{proof}
\begin{notation}
For two endomorphisms $f_X$ and $g_Y$ the notation $f_X \eiso g_Y$ will indicate that there exists an isomorphism
$\alpha: X \xra{\sim} Y$ such that $\alpha_\ast(f_X)=g_Y$.
\end{notation}
\begin{remark}
	\label{remark:up_to_iso}
	The above lemma implies that in $R_\Phi$ endomorphisms of $X$ are identified with endomorphisms of $Y$ via all
	isomorphisms $X \xra{\sim} Y$.
	In particular, an endomorphism $f_X$ is identified 
	with the ``twisted'' version $\sigma \circ f \circ \sigma^{-1}$ for every automorphism $\sigma : X \ra X$.
Because of these identifications, an essentially equivalent approach to the construction of $R_\Phi$ could be obtained by
taking $\Phi$ to be a set of isomorphism classes of objects closed under the $\tens$-product.
However, such an approach would obscure the fact that these identifications up to isomorphism are forced by the innocuous identifications
	$f \sim a \tens f$ and $f \sim f \tens a$.
\end{remark}
\begin{lemma}
	\label{lemma:concrete_equivalence}
	Two endomorphisms $f_X$ and $g_Y$ are equivalent
	in $R_\Phi$ if and only if
there exist objects $a,b \in \Phi$ such 
that $a \tens f_X \eiso b \tens g_Y$ if and only if there exists an object $c \in \Phi$
such that $f_X \tens c \tens Y = X \tens c \tens g_Y$.
\end{lemma}
\begin{proof}
To show that the equivalence relation defined by the third condition
(existence of $c\in \Phi$ such that $f_X
	\tens c \tens Y = X \tens c \tens g_Y$)
is stronger 
than the equivalence relation defining $R_\Phi$ we need to show that $f \in E_X$ is identified with
	$a \tens f$ and $f \tens a$ for any $a \in \Phi$. Indeed, the fact that $f_X$ is $\tens$-balanced implies
	that $f \tens c \tens a \tens X = X \tens c \tens a \tens f$ and $f \tens c \tens X \tens a = X \tens c\tens f \tens a$ 
	for \emph{any} object~$c\in\K$ (cf.~the proof of Lemma~\ref{lemma:twisting_lemma}).
	On the other hand, 
	$f_X \tens c \tens Y = X \tens c \tens g_Y$ clearly implies that $[f_X] = [g_Y]$ in $R_\Phi$.
Similarly, the second condition defines an equivalence relation which is evidently stronger than the one defining $R_\Phi$.
	On the other hand, if $a \tens f_X \simeq b\tens g_Y$ then by definition there exists an isomorphism $\alpha:a \tens X \xra{\sim} b\tens Y$
	such that $b \tens g_Y = \alpha_\ast(a\tens f_X)$. 
	Lemma~\ref{lemma:iso_forced} then implies that $a\tens f_X \tens b \tens Y = a \tens X \tens b \tens g_Y$ so that $[f_X] = [g_Y]$ in $R_\Phi$.
\end{proof}
\begin{proposition}
	The set $R_\Phi$ is a commutative ring with addition and multiplication defined by
$[f_X]+[g_Y] := [f_X \tens Y + X \tens g_Y]$ and $[f_X]\cdot [g_Y] := [(f_X \tens Y)\circ (X \tens g_Y)]$.
It is non-zero provided that $\Phi$ does not contain 
a zero object.
	If $(0)$ is a prime in~$\K$, for example if $\K$ is rigid and local, then $R_\Phi$ is a local ring provided that it is \mbox{non-zero}.
\end{proposition}
\begin{proof}
	Armed with Lemma~\ref{lemma:concrete_equivalence}, it is a long but straightforward exercise to establish that addition and multiplication are well-defined and 
	endow  $R_\Phi$ with a ring structure.
	There are several ways to see that 
	this ring structure is commutative.
	For example, the fact that $f_X$ is $\tens$-balanced implies that
	$f \tens g \tens X = X \tens g \tens f$ and
	so $[f]\cdot [g] = [f\tens g] = [f\tens g\tens X]=[X \tens g \tens f] = [g\tens f]=[g]\cdot [f]$.
	One readily checks that $[f] = 0$ iff $a \tens f = 0$ for some $a \in \Phi$ and that $[f]$ is a unit iff $a \tens f$ is an isomorphism for some $a \in \Phi$.
	The proof that $R_\Phi$ is local when $(0)$ is prime closely mirrors the proof that $E_X$ is local (Proposition~\ref{proposition:local_ring}).
	Indeed, 
	if $[f_X] + [g_Y]$ 
	is a unit in $R_\Phi$ then $a \tens (f \tens Y + X \tens g)$ is an isomorphism for some $a \in \Phi$.
	It follows that $a^{\tens n} \tens (f \tens Y + X \tens g)^{\tens n} \tens \cone(f) \tens \cone(g)$ is both zero and an isomorphism for $n\ge 3$.
This implies that
$a^{\tens n}\tens (X \tens Y)^{\tens n} \tens \cone(f) \tens \cone(g) = 0$
for $n \ge 3$ and since $(0)$ is prime and $X,Y,a \in \Phi$ are non-zero we conclude that $f$ or $g$ is an isomorphism (and hence $[f]$ or $[g]$ is a unit in $R_\Phi$).
\end{proof}
\begin{remark}
	\label{remark:E_X_colimit}
	The ring $R_\Phi$ is the colimit of a diagram of rings consisting of $E_X$ for each $X \in \Phi$ with maps
	generated by $a\tens -:E_X \ra E_{a\tens X}$ and $-\tens a:E_X \ra E_{X \tens a}$.
	Although the index category on which this diagram is defined is not technically a filtered category (because there are parallel arrows that
	are not coequalized in the category), $\colim_{X\in \Phi} E_X$
	is still a filtered colimit in the more general sense of \cite[Chapter 4, Section 17]{stacks_project}
	which is sufficient for the colimit to be created in the category of sets. 
	Rather than define $R_\Phi$ to be this colimit (which would necessitate a longer discussion of these technicalities) we have opted for the concrete description
	given above.
\end{remark}
\begin{proposition}
	\label{proposition:general_induced}
	Let $F: \K \ra \L$ be a morphism of tensor triangulated categories.
	Suppose $\Phi \subset \K$ and $\Psi \subset \L$ are non-empty
	$\tens$-multiplicative subsets such that $F(\Phi) \subset \Psi$. 
	Then $[f] \mapsto [F(f)]$ defines a ring homomorphism
	$R_{\K,\Phi} \ra R_{\L,\Psi}$.
\end{proposition}
\begin{proof}
	This is a routine verification using Lemma~\ref{lemma:concrete_equivalence} and the fact that $F: \K \ra \L$ is a strong $\tens$-functor.
\end{proof}
\begin{theorem}\label{theorem:general_ungraded_comparison_map}
	Let $\K$ be a tensor triangulated category and let $\Phi
	\subset \K$ be a non-empty set of objects closed under the
	$\tens$-product.  
	Let $\cZ_\Phi := \bigcap_{X \in \Phi} \supp(X)$.
	There is an inclusion-reversing, spectral
	map
	\[ \rho_{\Phi} : \cZ_\Phi \ra \Spec(R_{\Phi})\]
	defined by $\P \mapsto \{ [f] \in R_{\Phi} \mid \cone(f) \notin \P\}$.
\end{theorem}
\begin{proof}
	The first point to make is that for $\P \in \cZ_\Phi$ the condition $\cone(f) \notin \P$ does not depend on the choice of
	representative of $[f] \in R_\Phi$.
	Indeed, if $[f] = [g]$ then $a \tens f \simeq b \tens g$ for some $a,b\in \Phi$, so $\cone(a \tens f) \simeq \cone(b \tens g)$.
	If $\P \in \cZ_\Phi$ then 
	$\P \in \supp(\cone(f))$ iff $\P \in \supp(a)\cap \supp(\cone(f)) = \supp(b) \cap \supp(\cone(g))$
	iff $\P \in \supp(\cone(g))$.
	The second point to make is that for any prime $\P \in \Spc(\K)$ the quotient functor $q:\K \ra \K/\P$ induces
	a ring homomorphism $R_{\K,\Phi} \ra R_{\K/\P,q(\Phi)}$ whose target ring is local provided that $0 \notin q(\Phi)$; 
	in other words, provided that $\P \cap \Phi =\emptyset$ which is equivalent to saying that $\P \in \cZ_\Phi$.
	With these facts in mind the proof is similar to the proof of Theorem~\ref{theorem:unnatural_ungraded_comparison_map}.
	For an element $[f] \in R_{\K,\Phi}$, $[q(f)]$ is a unit in $R_{\K/\P,q(\Phi)}$ iff
	there exists $a \in \Phi$ such that $q(a\tens f)$ is an isomorphism in $\K/\P$ iff there exists $a\in \Phi$ such that $a \tens \cone(f) \in \P$ iff $\cone(f) \in \P$.
	Thus $\rho_\Phi(\P)$ is the pullback of the collection of non-units in $R_{\K/\P,q(\Phi)}$ and since the non-units in a local ring from a (two-sided) ideal, 
	this establishes that $\rho_\Phi(\P)$ is an ideal.
	It is prime since $[f_X]\cdot [g_Y] \in \rho_\Phi(\P)$ 
	implies that $\P \in \supp(\cone(f\tens g)) \subset
	\supp(\cone(f \tens Y))\cup\supp(\cone(X \tens g))
	\subset \supp(\cone(f)) \cup \supp(\cone(g))$.
	That $\rho_\Phi$ is a spectral map is established by similar modifications to the argument given in the proof of Theorem~\ref{theorem:unnatural_ungraded_comparison_map}.
\end{proof}
\begin{example}
	\label{example:object_comparison_maps}
	For any object $X \in \K$, taking $\Phi := \{ X^{\tens n} \mid n \ge 1\}$
	provides a comparison map
	$\rho_X : \supp(X) \ra \Spec(R_X)$.
	These are the ``object'' comparison maps mentioned in the introduction.
\end{example}
\begin{example}
	\label{example:closed_set_comparison_maps}
	For any closed set $\cZ \subset \Spc(\K)$,
	taking 
	$\Phi := \{ a \in \K \mid \supp(a) \supset \cZ \}$
	gives a comparison map
	$\rho_\cZ : \cZ \ra \Spec(R_\cZ)$.
	These are the ``closed set'' comparison maps mentioned in the introduction.
	Note that $\cZ_\Phi = \cZ$ because 
	$\{\supp(a) : a\in \K\}$ forms a basis of closed sets for the topology on $\Spc(\K)$.
Also note that if the category $\K$ is not small then there is the unfortunate detail that
$\Phi$ might not be a set; however, we do not need to worry about this technicality because of
Remark~\ref{remark:up_to_iso}.
\end{example}
\begin{example}\label{example:thomason_closed_subset_comparison_maps}
	For any Thomason closed set $\cZ \subset \Spc(\K)$,
	taking
	$\Phi := \{ a \in \K \mid \supp(a) = \cZ \}$
	provides another comparison map defined on $\cZ$.
	However, the ring $R_\Phi$ is canonically
	isomorphic to the one in Example~\ref{example:closed_set_comparison_maps}
	and under this identification the two comparison maps coincide.
	In other words, when $\cZ$ is Thomason we can take the target ring of
	the ``closed set'' comparison map $\rho_\cZ : \cZ \ra \Spec(R_\cZ)$ to
	be defined using only those objects $X$ for which $\supp(X) = \cZ$.
\end{example}
\begin{example}
	Another candidate to consider would be 
	$\Phi := \{ a \in \K | \supp(a) \subset \cZ\}$
	but in this case	
	 $\cZ_\Phi =
	\emptyset$ and $R_\Phi = 0$.
	Indeed, it is easily checked that $R_\Phi = 0$ iff $0 \in \Phi$ iff $\Phi$ contains an object that is $\tens$-nilpotent iff
	$\cZ_\Phi = \emptyset$.
\end{example}
\begin{remark}
	There are other examples that could be considered, such as
	the collection of $\tens$-invertible objects, or 
	the collection of objects that are isomorphic to a direct sum of suspensions of $\unit$.
\end{remark}
\begin{remark}
	For any commutative ring $A$ and ring homomorphism $A \ra R_\Phi$ one obtains an inclusion-reversing, spectral map
	$\cZ_\Phi \ra \Spec(A)$ by composing $\rho_\Phi$ with the induced map $\Spec(R_\Phi) \ra \Spec(A)$.
	For example, the comparison maps $\rho_{X,A}:\supp(X) \ra \Spec(A)$ defined in Section~\ref{section:basic_construction} are recovered
	from $\rho_X:\supp(X) \ra \Spec(R_X)$ by composing $A \ra E_X$ with the canonical map $E_X \ra R_X$.
	In fact, the ring $R_\Phi$ can be regarded as a colimit of all the commutative rings $A$ mapping into the 
	rings $E_X$ for $X \in \Phi$.
	More precisely, consider triples $(A,\alpha,X)$ where $A$ is a commutative ring, $\alpha : A \ra E_X$ is a ring homomorphism, and $X$ is an object of $\Phi$.
	Define a morphism $(A,\alpha,X) \ra (A',\alpha',X')$ to be a morphism $u:A \ra A'$ such that 
	\[\xymatrix @R=0.5em{
			A \ar[dd]_-{u} \ar[r]^-{\alpha} & E_X \ar[dr] \\
			& & E_{X\tens a \tens X'}\\
			A' \ar[r]^-{\alpha'} & E_{X'} \ar[ur]
		}\]
	commutes for some object $a \in \Phi$. (If the object $a$ were not included then one would run into difficulties composing such morphisms.)
There is an obvious functor $(A,\alpha,X) \mapsto A$ from the category of triples to the category of rings and it is easy to check
that the maps $A \xra{\alpha} E_X \ra R_\Phi$ induce a morphism \mbox{$\colim_{(A,\alpha,X)}A \ra R_\Phi$}.
Let us briefly sketch the proof that this is an isomorphism. For surjectivity, one observes that if $[f_X] \in R_\Phi$
then the subring $\Z[f_X] \subset E_X$ generated by $f_X$ is commutative and provides a triple $(\Z[f_X],i,X)$.
On the other hand, to establish injectivity one first proves that the category of triples is filtered. It follows that
every element of $\colim_{(A,\alpha,X)} A$ is the image of an element $x\in A$ under the canonical map
$A \ra \colim_{(A,\alpha,X)} A$ associated with a triple $(A,\alpha,X)$.
If the element in $\colim_{(A,\alpha,X)} A$ goes to zero in $R_\Phi$ then $[\alpha(x)] = 0$ in $R_\Phi$ and so 
there is some $a \in \Phi$
such that $a \tens \alpha(x) = 0$.
If $\beta$ is the composite $A \xra{\alpha} E_X \ra E_{a \tens X}$ then the induced map $\overline{\beta}:A/\ker \beta \ra E_{a \tens X}$
provides a triple $(A/\ker \beta,\overline{\beta},a \tens X)$ while the quotient map $\pi:A \ra A/\ker \beta$ defines a map of
triples $(A,\alpha,X) \ra (A/\ker \beta, \overline{\beta},a \tens X)$. Since the canonical map $A \ra \colim_{(A,\alpha,X)}$
for the triple $(A,\alpha,X)$ factors through the above map of triples, we conclude that the image of $x$ in $\colim_{(A,\alpha,X)} A$ is zero.
From this perspective, the maps defined in Section~\ref{section:basic_construction}
are obtained from $\rho_\Phi$ by pulling back via the canonical map $A \ra
\colim_{(A,\alpha,X)} A \simeq R_\Phi$ associated to a triple~$(A,\alpha,X)$.
\end{remark}
\begin{proposition}
	Let $F:(\K,\Phi) \ra (\L,\Psi)$ be a morphism of tensor triangulated categories $\K \ra \L$ such that $F(\Phi)\subset\Psi$.
Then there is a commutative diagram
\begin{equation}
	\label{diagram:naturality_of_general_map}
	\begin{gathered}
	\xymatrix @C=0em{
			\Spc(\L)\ar[d] & \supset &  \cZ_{\L,\Psi} \ar[rrrrrrrr]^-{\rho_{\L,\Psi}} \ar[d] &&&&&&&& \Spec(R_{\L,\Psi}) \ar[d] \\
			\Spc(\K)              & \supset & \cZ_{\K,\Phi} \ar[rrrrrrrr]_-{\rho_{\K,\Phi}} &&&&&&&& \Spec(R_{\K,\Phi})}
	\end{gathered}
\end{equation}
in the category of spectral spaces.
\end{proposition}
\begin{proof}
	A ring homomorphism $R_{\K,\Phi} \ra R_{\L,\Psi}$ is provided by Proposition~\ref{proposition:general_induced} and the rest is a routine recollection of the
	relevant definitions.
\end{proof}
\begin{example}\label{example:inclusion_of_closed_subsets}
	If $\cZ_1 \subset \cZ_2$ is an inclusion of closed subsets then there is a ring homomorphism $R_{\cZ_2} \ra R_{\cZ_1}$ and a commutative diagram
	\[\xymatrix @R=2em{
			\cZ_2 \ar[r]^-{\rho_{\cZ_2}} &\Spec(R_{\cZ_2}) \\
\cZ_1 \ar@{}[u]|-*[@]{\subset} \ar[r]^-{\rho_{\cZ_1}} & \Spec(R_{\cZ_1}). \ar[u]
		}\]
\end{example}
\begin{example}\label{example:from_thomason_to_object}
	If $\cZ$ is a Thomason closed subset
	then for any object $X$ with $\supp(X) = \cZ$ there is a 
	ring homomorphism $R_X \ra R_{\cZ}$ and a commutative diagram
	\[\xymatrix{
			\cZ \ar[r]^-{\rho_\cZ} \ar[dr]_-{\rho_X} & \Spec(R_\cZ) \ar[d] \\
			& \Spec(R_X).
		}\]
\end{example}
It is worth explicitly stating the naturality in the case of the object and closed set comparison maps:
\begin{proposition}
	If $F:\K \ra \L$ is a morphism of tensor triangulated categories and $X \in \K$ 
	then there is a commutative diagram
	\[\xymatrix @C=0em{
			\Spc(\L)\ar[d] & \supset & \supp_\L(FX) \ar[rrrrrrrr]^-{\rho_{\L,FX}} \ar[d] &&&&&&&& \Spec(R_{\L,FX}) \ar[d] \\
			\Spc(\K)              & \supset & \supp_\K(X) \ar[rrrrrrrr]_-{\rho_{\K,X}} &&&&&&&& \Spec(R_{\K,X})
		}\]
	in the category of spectral spaces, where the left square is cartesian. 
\end{proposition}
\begin{proposition}
	If $F:\K \ra \L$ is a morphism of small tensor triangulated categories and $\cZ \subset \Spc(\K)$ is a closed subset then there is a commutative diagram
	\[\xymatrix @C=0em{
			\Spc(\L)\ar[d]_{\phi} & \supset &  \phi^{-1}(\mathcal{Z}) \ar[rrrrrrrr]^-{\rho_{\L,\phi^{-1}(\mathcal{Z})}} \ar[d] &&&&&&&& \Spec(R_{\L,\phi^{-1}(\mathcal{Z})}) \ar[d] \\
			\Spc(\K)              & \supset & \mathcal{Z} \ar[rrrrrrrr]_-{\rho_{\K,\mathcal{Z}}} &&&&&&&& \Spec(R_{\K,\mathcal{Z}})
		}\]
	in the category of spectral spaces, where the left square is cartesian.
\end{proposition}
\begin{remark}
	Considering $\supp_\K(X)$ and $\Spec(R_{\K,X})$ as contravariant functors
	from the category of essentially small tensor triangulated categories with
	chosen object to the category of spectral spaces, the object comparison
	maps $\rho_{\K,X}$
	can be regarded 
	as a natural transformation
	$\supp_{\K}(X) \ra \Spec(R_{\K,X})$.
	Similarly, there is a contravariant ``forgetful'' functor
	$(\K,\cZ) \mapsto \cZ$	
	from the category
	of small tensor triangulated categories with chosen closed subset of their
	spectrum to the category of spectral spaces, and the closed set comparison maps 
	$\rho_{\K,\cZ}$
	form a natural transformation
	from this functor to the functor $(\K,\cZ) \mapsto\Spec(R_{\K,\cZ})$.
	Finally,  $\rho_{\K,\Phi}$ can be regarded
	as a natural transformation from
	$(\K,\Phi) \mapsto \cZ_\Phi$ to 
	$(\K,\Phi) \mapsto \Spec(R_{\K,\Phi})$.
\end{remark}
\begin{remark}
It is straightforward to develop the graded version of these constructions.
One checks that the graded analogue of Corollary~\ref{corollary:tensor_Ex} holds and then defines
$R_\Phi^\bullet$ to be the colimit of the diagram
of graded rings generated by the maps
$E_X^\bullet \ra E_{X\tens Y}^\bullet$ and $E_Y^\bullet \ra E_{X\tens Y}^\bullet$. 
One checks that this is a filtered colimit (in the weak sense---see Remark~\ref{remark:E_X_colimit})
and it is easily determined how filtered colimits of graded rings are constructed.
To be clear, the abelian group $R_\Phi^i$ is the filtered colimit of abelian groups $\colim_{X\in\Phi} E_X^i$
and thus
consists of equivalence classes $[f]$ where $f \in E_X^i$ for some $X \in \Phi$.
The product on $R_\Phi^\bullet$ is given by
\[\xymatrix @R=0.15em{
		\colim_{X\in \Phi} E_X^i \times \colim_{Y\in\Phi} E_Y^j \ar[r] & \colim_{Z\in\Phi}E_Z^{i+j} \\
		([f_X],[g_Y]) \ar@{|->}[r] & [(f_X\tens Y)\cdot (X\tens g_Y)]}\]
where $(f_X\tens Y)\cdot(X \tens g_Y)$ is the graded product in $E_{X\tens Y}^\bullet$.
Note that $R_\Phi^0$ is exactly the ungraded ring $R_\Phi$ from Definition~\ref{definition:general_ring}.
It is straightforward to show that $R_\Phi^\bullet$ is graded-commutative
although one needs to be clear 
about our abuses of notation
concerning the suspension isomorphisms. 
Ultimately the graded-commutativity comes from the anti-commutativity of
diagram~\eqref{diagram:anticommutes} in the axioms of a tensor triangulated category. 

The proof of the following theorem is very similar to the proof of Theorem~\ref{theorem:graded_comparison_map} just with the kind of modifications we saw in the proof
of Theorem~\ref{theorem:general_ungraded_comparison_map}.
\end{remark}
\begin{theorem} Let $\K$ be a tensor triangulated category and let $\Phi \subset \K$ be a non-empty set of objects closed under
	the $\tens$-product. There is a graded-commutative graded ring $R_\Phi^\bullet$ and an inclusion-reversing, spectral map
	\[ \rho_\Phi^\bullet : \cZ_\Phi \ra \Spech(R_\Phi^\bullet) \]
defined by $\rho_\Phi^\astt(\P) := \{[f] \in R_\Phi^i \mid \cone(f) \notin \P \}_{i \in \Z}$.
	The ring $R_\Phi$ is precisely $R_\Phi^0$ and $\p^\bullet \mapsto \p^\bullet \cap R_\Phi^0$ defines
	a surjective spectral map $\Spech(R_\Phi^\bullet) \ra \Spec(R_\Phi)$ such that the following diagram commutes
	\[\xymatrix{ 
			\cZ_\Phi \ar[dr]_{\rho_\Phi} \ar[r]^-{\rho_\Phi^\bullet} & \Spech(R_\Phi^\bullet) \ar[d]^{(-)^0} \\
			& \Spec(R_\Phi).
		}\]
\end{theorem}
\begin{remark}
	The graded comparison maps have the same kind of naturality properties as the ungraded comparison maps.
\end{remark}

\section{Object comparison maps}
In this section we will establish some of the basic features of the 
natural ``object'' comparison maps $\rho_X:\supp(X) \ra \Spec(R_X)$ defined in Example~\ref{example:object_comparison_maps}.
More specifically, our primary goal is to establish that $\rho_X$ is invariant under some natural operations that can be performed on the object $X$ such as taking duals, or suspensions, or $\tens$-powers, etc.
Before we begin proving such results, let us remark that for $X = \unit$ the canonical
map $[\unit,\unit]=E_\unit \ra R_\unit$ is an isomorphism and under this identification \mbox{$\rho_\unit : \Spc(\K) \ra \Spec(R_\unit)$} is the original unit comparison map 
from \cite{Balmer_SSS}; similarly for the graded version.
It will also be convenient to recognize that the canonical homomorphisms $E_{X^{\tens n}} \ra R_X$ induce an isomorphism
\begin{equation}
	\label{eq:object_colimit}
	\xymatrix@1{
		R_X \simeq \colim (E_X
		\ar[r]^-{X \tens -} &
		E_{X^{\tens 2}}
		\ar[r]^-{X \tens -} &
		E_{X^{\tens 3}}
		\ar[r]^-{X \tens -} &
		E_{X^{\tens 4}}
		\ar[r]^-{X \tens -} &
		\cdots)
	}
\end{equation}
and we will often tacitly make the identification $R_X = \colim_{n \ge 1} E_{X^{\tens n}}$.
\begin{lemma}
	\label{lemma:Nf}
	If $f \in E_X$ then $\langle \cone(f) \rangle \subset \{ a \in \K \mid a \tens f^{\tens n} = 0 \text{ for some } n \ge 1\}$.
\end{lemma}
\begin{proof}
Standard techniques verify that the right-hand side is a thick $\tens$-ideal and
the inclusion then follows from Lemma~\ref{lemma:killscone}; cf.~\cite[Theorem 2.15]{Balmer_SSS}.
\end{proof}
\begin{proposition}
An isomorphism $\alpha : X \xra{\sim} Y$ in $\K$ induces an isomorphism of rings
$\alpha_* : R_X \xra{\sim} R_Y$.
	Under this identification, $\rho_X$ coincides with $\rho_Y$.
\end{proposition}
\begin{proof}
	This is routine from the definitions; cf.~Lemma~\ref{lemma:iso_forced}.
\end{proof}
\begin{proposition}\label{proposition:tensor_maps}
Tensoring on the right by an object $Y$ induces a
	ring homomorphism $R_X \ra
	R_{X\tens Y}$ and a commutative diagram 
\begin{equation}
	\label{diagram:XtensY}
	\begin{gathered}
	\xymatrix @C=4em{ 
			\supp(X) \ar[r]^-{\rho_X} & \Spec(R_X) \\
			\supp(X) \cap \supp(Y) \ar@{^{(}->}[u] \ar[r]^-{\rho_{X\tens Y}} & \Spec(R_{X\tens Y}). \ar[u]
		}
	\end{gathered}
\end{equation}
	If $\supp(X) \subset \supp(Y)$ then the
	kernel of the map $R_X \ra R_{X \tens Y}$ consists entirely of
	nilpotents.
There is a similar result for tensoring on the left.
\end{proposition}
\begin{proof}
	Note that there is a canonical
	isomorphism $X^{\tens n} \tens Y^{\tens n}\xra{\sim}(X\tens Y)^{\tens n}$ obtained
	from the symmetry that preserves the
	order of the $X$'s and the order of the
	$Y$'s. 
	One can then define a ring homomorphism $E_{X^{\tens n}} \ra E_{(X \tens Y)^{\tens n}}$ as the composition of $-\tens Y^{\tens n} : E_{X^{\tens n}} \ra E_{X^{\tens n} \tens Y^{\tens n}}$
	and the isomorphism $E_{X^{\tens n} \tens Y^{\tens n}} \xra{\sim} E_{(X\tens Y)^{\tens n}}$ and
	one readily verifies that these maps
	induce a homomorphism $R_X \ra R_{X\tens Y}$.
That \eqref{diagram:XtensY} commutes follows
from the definitions, observing that if $\P
\in \supp(Y)$ then $a \in \P$ iff $a \tens Y
\in \P$.
Finally, if $[f] \in R_X$ 
is mapped to zero in
$R_{X\tens Y}$ 
then
$X^{\tens i} \tens Y^{\tens j} \tens f = 0$ for some $i,j\ge 1$.
It follows using the condition $\supp(X) \subset \supp(Y)$ that
$\supp(X) = \supp(\cone(f))$ and hence that 
$X^{\tens k} \in \langle \cone(f) \rangle$ for some $k \ge 1$.
Lemma~\ref{lemma:Nf} then implies that $X^{\tens k} \tens f^{\tens n} = 0$ for some $n \ge 1$
and we conclude that $[f]$ is a nilpotent element of~$R_X$.
\end{proof}
\begin{proposition} 
	\label{proposition:XkY}
For every pair of objects $X$ and $Y$ and every integer $k \ge 1$, 
tensoring on the left by $X^{\tens (k-1)}$ induces an isomorphism 
$R_{X\tens Y} \xra{\sim} R_{X^{\tens k} \tens Y}$.
Under this identification the maps $\rho_{X\tens Y}$ and $\rho_{X^{\tens k} \tens Y}$ coincide.
\end{proposition}
\begin{proof}
Since the homomorphisms $R_{X\tens Y} \hookrightarrow R_{X^{\tens 2}\tens Y} \hookrightarrow R_{X^{\tens 3}\tens Y} \hookrightarrow \cdots$
induced by $X \tens -$ are evidently injective, the problem reduces to showing that $R_{X \tens Y} \ra R_{X^{\tens 2} \tens Y}$ is surjective.
	In other words, we need to show that every $f \in E_{(X^{\tens 2}\tens Y)^{\tens n}}$ is equivalent in $R_{X^{\tens 2}\tens Y}$
	to an element coming from $R_{X\tens Y}$.
	We'll give the proof under the assumption that $n=1$. The proof for arbitrary $n \ge 1$ is similar.

	Consider the element $g \in E_{(X\tens Y)^{\tens 3}}$ corresponding to
	$X \tens f \tens Y^{\tens 2} \in E_{X^{\tens 3} \tens Y^{\tens 3}}$.
	Under the map $R_{X \tens Y} \ra R_{X^{\tens 2} \tens Y}$, $[g]$ is sent to the image in $R_{X^{\tens 2} \tens Y}$ of the element 
	in $E_{(X^{\tens 2} \tens Y)^{\tens 3}}$ corresponding to $X^{\tens 4} \tens f \tens Y^{\tens 2} \in E_{X^{\tens 6} \tens Y^{\tens 3}}$.
	We claim that this element in $E_{(X^{\tens 3} \tens Y)^{\tens 3}}$ equals $(X^{\tens 2} \tens Y)^{\tens 2} \tens f$ so that
	the image of $[g]$ in $R_{X^{\tens 2} \tens Y}$ equals $[f]$.
	This is not completely obvious and involves an unilluminating trick.
In order to describe this trick, let's
write $a:= X^{\tens 2}$ and $b := Y$ for simplicity of notation; so $f \in E_{a\tens b}$.
We will use subscripts to indicate position and we'll drop the tensors from the notation.
Consider the diagram
\[\xymatrix @C=8em{
		a_1  b_1 a_2 b_2 a_3 b_3
		\ar[d]^{\sim}
		\ar@<0.25em>[r]^-{a b a b f} 
		\ar@<-0.25em>[r]_-{a b f a b}
		&
		a_1  b_1 a_2 b_2 a_3 b_3 \ar[d]^{\sim}
		\\
		a_1 a_2 b_2 a_3 b_1 b_3
		\ar@<0.25em>[r]^-{a f a b^2}
		\ar@<-0.25em>[r]_-{a^2 b f b} \ar[d]^{\sim} & 
		a_1 a_2 b_2 a_3 b_1 b_3 \ar[d]^{\sim}
		\\
		a_1 a_2 a_3 b_1 b_2 b_3 \ar[r]^-{a^2 f b^2} &
		a_1 a_2 a_3 b_1 b_2 b_3 
	}\]
where the vertical maps are the indicated permutations of the factors induced by the symmetry.
Note that the composition of the two vertical permutations is the unique permutation from the source to the target that preserves the order of the $a$'s and the order of the $b$'s.
The top arrow is $(X^2 \tens Y)^{\tens 2}\tens f$ and the bottom arrow is the morphism that it
corresponds to in $E_{X^{\tens 6} \tens
	Y^{\tens 3}}$. The commutativity of the
diagram verifies that this equals $X^{\tens 4} \tens f \tens Y^{\tens 2}$ as claimed.
\end{proof}
\begin{proposition}
	If $\K$ is a rigid tensor triangulated category then for every object $X$ in $\K$ there is a
	canonical isomorphism $R_X \simeq R_{DX}$ under which the map $\rho_X$ coincides with $\rho_{DX}$.
\end{proposition}
\begin{proof}
	The duality functor $D:\K^{\op} \ra \K$ gives a ring isomorphism $[X,X] \xra{\sim} [DX,DX]^{\op}$ and an easy application of the fact that 
	$D$ is a strong $\tens$-functor shows that the isomorphism restricts to an isomorphism $E_X \xra{\sim} E_{DX}^{\op}$.
	It is straightforward but tedious to verify that these isomorphisms induce an isomorphism $R_X \xra{\sim} R_{DX}^{\op} = R_{DX}$.
	Showing that $\rho_X$ and $\rho_{DX}$ correspond amounts to
	showing that $\supp(\cone(Df)) =
	\supp(D(\cone(f)))$. Here we use the fact
	that $D$ is an exact functor of
	triangulated categories and the fact that
	$\supp(DX) =
	\supp(X)$ in any rigid category.
\end{proof}
Our next goal is to establish that $\rho_X = \rho_{X \oplus X}$.
Note that under the usual identification of $(X \oplus X)^{\tens n}$ with a $\oplus$-sum of $2^n$ copies of $X^{\tens n}$ an endomorphism $(X\oplus X)^{\tens n} \ra (X \oplus X)^{\tens n}$ can be regarded as a $2^n \times 2^n$ matrix $(f_{ij})$ with entries $f_{ij} : X^{\tens n} \ra X^{\tens n}$.
We will make such identifications without further comment.
\begin{lemma}
	\label{lemma:direct_sum_characterization}
	An endomorphism $f=(f_{ij}) : a_1 \oplus \cdots \oplus a_n \ra a_1 \oplus \cdots \oplus a_n$ is contained in $E_{a_1 \oplus \cdots \oplus a_n}$ if and only if
	\begin{enumerate}
		\item $a_i \tens f_{jj} = f_{ii}\tens a_j$ for all $i,j$, and
		\item $f_{ij}\tens (a_1 \oplus \cdots \oplus a_n) = 0$ for $i\neq j$.
	\end{enumerate}
\end{lemma}
\begin{proof}
	Observe that $f\tens (a_1 \oplus \cdots
	\oplus a_n)$, viewed as an $n^2 \times
	n^2$ matrix, consists of $n\times n$
	blocks, each of which is diagonal, while
	$(a_1 \oplus \cdots \oplus a_n) \tens f$
	consists of $n\times n$ blocks, arranged
	along the diagonal.  Equating the
	off-diagonal blocks gives the condition
	that $f_{ij} \tens a_k=0$ for all $k$ if
	$i \neq j$, which is equivalent to
	condition~(2). 
	Similarly, equating the off-diagonals of
	the diagonal blocks 
	gives the equivalent condition that
	$a_k \tens f_{ij} = 0$ for all $k$ if $i \neq j$.
	On the other hand, the diagonal
	of the $i$th diagonal block 
	gives the condition that
	$f_{ii}\tens a_j = a_i \tens f_{jj}$ for
	all $j$.
\end{proof}
\begin{corollary}
	\label{corollary:direct_sum_surjective}
	If $f \in E_{(X\oplus X)^{\tens n}}$ then there exists some $\alpha \in
	E_{X^{\tens n}}$ such that ${(X\oplus X)^{\tens n}} \tens f$ (regarded as a
	matrix of endomorphisms of $X^{\tens 2n}$) is diagonal with copies of
	$X^{\tens n} \tens \alpha$ along the diagonal.
\end{corollary}
\begin{proof}
	This follows from
	Lemma~\ref{lemma:direct_sum_characterization}.
	We can take $\alpha := f_{11}$ for
	example.
\end{proof}
\begin{proposition}
	\label{proposition:object_direct_sum}
	Let $\K$ be a tensor triangulated category and let $X$ be an object in~$\K$.
	There is a canonical isomorphism $R_X \simeq R_{X \oplus X}$ under which $\rho_X$ coincides with $\rho_{X\oplus X}$.
\end{proposition}
\begin{proof}
Invoking Lemma~\ref{lemma:direct_sum_characterization},
we see that there is a ring homomorphism
	$\Delta: E_{X^{\tens n}} \hookrightarrow E_{(X \oplus X)^{\tens n}}$ for each $n \ge 1$
	which sends
	$f\in E_{X^{\tens n}}$ to the diagonal matrix consisting of copies of $f$ on the diagonal.
	One checks that these maps commute with $X \tens -$ and $(X\oplus X)\tens -$ and therefore induce an injection
	$R_X \hookrightarrow R_{X\oplus X}$.
	Surjectivity of this map follows from Corollary~\ref{corollary:direct_sum_surjective}.
That $\rho_X$ and $\rho_{X\oplus X}$ correspond boils down to the definitions and the fact that $\cone(f\oplus \cdots \oplus f)\simeq \cone(f) \oplus \cdots \oplus \cone(f)$.
\end{proof}
A similar argument shows that $\rho_X = \rho_{X \oplus \Si X}$ after a canonical identification $R_X \simeq R_{X\oplus \Si X}$.
More generally:
\begin{proposition}
	\label{proposition:oplus_suspensions}
	Let $\K$ be a tensor triangulated category and let $X$ be an object of $\K$.
	If $Y$ is a $\oplus$-sum of suspensions of $X$ then $\rho_X = \rho_Y$ after a canonical identification $R_X \simeq R_Y$.
\end{proposition}
\begin{proof}
	The proof is a more advanced version of the proof of Proposition~\ref{proposition:object_direct_sum}.
	Observe that $(\Si^{i_1} X \oplus \cdots \oplus \Si^{i_k}X)^{\tens n}$
	may be identified with a $\oplus$-sum of suspensions of $X^{\tens n}$.
	One may define a ``diagonal'' map $E_{X^{\tens n}} \hookrightarrow E_{(\Si^{i_1} X \oplus \cdots \oplus \Si^{i_k} X)^{\tens n}}$
	which sends $f$ to a diagonal matrix whose diagonal entries are copies of $f$ suspended the appropriate numbers of times.
	One checks that these maps induce a map $R_X \hookrightarrow R_{(\Si^{i_1} X \oplus \cdots \oplus \Si^{i_k} X)}$
	and a similar argument shows that this map is in fact surjective.
\end{proof}
\begin{remark}
	There are graded versions of all of the above results, establishing that $\rho_X^\astt$ is invariant under suspension, tensor powers, and so on.
	The only result for which we should be careful is taking duals. 
The duality induces an isomorphism $R_X^\astt \simeq R_{DX}^{\astt,{\op}}$ but 
we can't remove the ${\op}$ because $R_{DX}^\astt$ is only graded-commutative.
	In any case, there is a canonical homeomorphism $\Spech(R_{DX}^{\astt,{\op}}) \simeq \Spech(R_{DX}^\astt)$
and under these identifications $\rho_X^\astt$ coincides with~$\rho_{DX}^\astt$.
\end{remark}
\begin{example}
For any object $X \in \K$ there is a homomorphism $[\unit,\unit]_\astt \ra E_X^\astt$ which sends $\alpha$
to $\alpha \tens X = X \tens \alpha$. These induce a homomorphism
$R_\unit^\astt \ra R_\Phi^\astt$ for any \mbox{$\tens$-multiplicative} set $\Phi \subset \K$ whose kernel consists of nilpotents (cf.~Proposition~\ref{proposition:tensor_maps}).
If $\Phi$ is taken to be the collection of objects that are isomorphic to direct sums of suspensions of $\unit$ then this map is an isomorphism (cf.~Proposition~\ref{proposition:oplus_suspensions}).
For example, if $\K = D^{\text{perf}}(k)$ for a field $k$, then every object is
a direct sum of suspensions of $\unit$ and the ``only'' comparison map is the
original unit comparison map $\rho_\unit^\astt : \Spc(\K) \ra
\Spech([\unit,\unit]_\astt)$.  More generally, it would be interesting to know whether
$R_\unit^\astt \ra
R_\Phi^\astt$ is an isomorphism (under suitable generation hypotheses) when
$\Phi$ is the collection of solid objects; in other words, whether the closed
set comparison map $\rho_{\Spc(\K)}^\astt$ associated with the whole spectrum reduces
to $\rho_\unit^\astt$.
\end{example}
\begin{remark}
Recall from the proof of Theorem~\ref{theorem:unnatural_ungraded_comparison_map} that under the unnatural comparison map
$\rho_{X,A} : \supp(X) \ra \Spec(A)$ the preimage of a Thomason closed subset $V(a_1,\ldots,a_n)\subset \Spec(A)$
is exactly the support of the ``Koszul'' object
$\cone(\alpha(a_1))\tens \cdots\tens \cone(\alpha(a_n))$.
On the other hand, for our natural comparison map $\rho_X:\supp(X) \ra \Spec(R_X)$
the elements of $R_X$ are equivalence classes of endomorphisms, but 
one still finds that
\[\rho_X^{-1}(V([f_1],\ldots,[f_n])) = \supp(\cone(f_1) \tens \cdots \tens \cone(f_n))\]
independent of the choice of representatives $f_i$.
Nevertheless, a different set of representatives $[f_1'],\ldots,[f_n']$ 
gives a different Koszul object $\cone(f_1')\tens \cdots \tens \cone(f_n')$ and there
is no reason \emph{a priori} for the comparison maps of these two
Koszul objects to coincide.
However,
$X^{\tens i} \tens \cone(f_1) \tens \cdots \tens \cone(f_n) \simeq 
X^{\tens j} \tens \cone(f_1') \tens \cdots \tens \cone(f_n')$
for some $i,j \ge 1$ and it follows from Proposition~\ref{proposition:XkY}
that the comparison map for $X \tens \cone(f_1) \tens \cdots \tens \cone(f_n)$
does not depend on the choice of representatives~$f_i$.
Thus when one decides to examine a closed set $\supp(X_0)$ more closely by
choosing generators of a Thomason closed subset $V([f_1],\ldots,[f_n]) \subset \Spec(R_{X_0})$, 
it is
advisable to take the generator of the preimage 
$\rho_{X_0}^{-1}(V([f_1],\ldots,[f_n]))$
to be
$X_1 := X_0 \tens \cone(f_1) \tens \cdots \tens \cone(f_n)$.
A serendipitous consequence of including $X_0$ as a $\tens$-factor is that
we then have a ring homomorphism $R_{X_0} \ra R_{X_1}$ and 
a commutative diagram
\begin{equation*}\label{diagram:next_step_object_method}
	\xymatrix @=1em {
		\supp(X_0) \ar[rr]^{\rho_{X_0}} && \Spec(R_{X_0}) \\
&& V([f_1],\ldots,[f_n]) \ar@{^{(}->}[u] \\
\supp(X_1) \ar[rr]^{\rho_{X_1}} \ar@{^{(}->}[uu] && \Spec(R_{X_1}). \ar[u]
	}
\end{equation*}
On the other hand, this procedure still apparently depends on the choice of
generators for the Thomason closed subset
$V([f_1],\ldots,[f_n])$.
\end{remark}

\subsection*{Idempotent completion}
Recall that every tensor triangulated category $\K$ may be embedded into an idempotent-complete tensor triangulated category $\K^\natural$ and that the
embedding $i: \K \hookrightarrow \K^\natural$ induces a homeomorphism of spectra (see~\cite[Remark~3.12]{Balmer_Spectrum}).
There is a precise sense in
which $\K$ and $\K^\natural$ admit ``the same'' theory of higher comparison maps.
We begin with the following unsurprising result.
\begin{proposition}
	\label{proposition:idempotent_invariant}
		For any non-empty $\tens$-multiplicative subset $\Phi \subset \K$, 
		there is a canonical identification $R_{\K,\Phi} \simeq
		R_{\K^\natural,i(\Phi)}$ while $\cZ_{\K,\Phi} \simeq
		\cZ_{\K^\natural,i(\Phi)}$ under the homeomorphism $i^* :
		\Spc(\K^\natural) \xra{\sim} \Spc(\K)$. Under these identifications,
		$\rho_{\K,\Phi} = \rho_{\K^\natural,i(\Phi)}$.  \end{proposition}
\begin{proof}
	This is a routine verification once one recalls all the definitions.
\end{proof}
In particular, this tells us that the object comparison maps for objects in $\K$ are unaffected when we pass to $\K^\natural$. 
But it still could be possible that in passing to $\K^\natural$ we get new object comparison maps coming from the new objects in $\K^\natural$.
However, this is not the case:
\begin{proposition}
	For any object $X \in \K^\natural$, the object $X \oplus \Si X$ is contained in $\K\subset \K^\natural$. There is 
	a canonical isomorphism $R_{\K^\natural,X} \simeq R_{\K,X\oplus \Si X}$ 
	and after this identification $\rho_{\K^\natural,X} = \rho_{\K,X\oplus \Si X}$.
\end{proposition}
\begin{proof} That $X \oplus \Si X$ is contained in $\K$ is a
	well-known fact; see the proof of
	\cite[Proposition~3.13]{Balmer_Spectrum} for example.
	Our result then follows from
	Proposition~\ref{proposition:idempotent_invariant} together
	with the result of
	Proposition~\ref{proposition:oplus_suspensions} which told us that $\rho_{X} = \rho_{X\oplus \Si X}$. 
\end{proof}
Next we can ask about the closed set comparison maps.
\begin{proposition}
	Let $\cZ$ be a closed subset of $\Spc(\K)$ and let $\cZ' = (i^*)^{-1}(\cZ)$ be the corresponding closed subset of $\Spc(\K^\natural)$.
	There is a canonical isomorphism $R_{\K,\cZ} \simeq R_{\K^\natural,\cZ'}$ such that with the identification $\cZ \simeq \cZ'$ the comparison map
	$\rho_{\K,\cZ}$ coincides with $\rho_{\K^\natural,\cZ'}$.
\end{proposition}
\begin{proof}
	Let $\Phi = \{x \in \K \mid \supp_\K(x) \supset \cZ\}$ and let $\Phi' = \{ a\in \K^\natural \mid \supp_{\K^\natural}(a) \supset \cZ'\}$.
	Then $i(\Phi) \subset \Phi'$ and we have an induced ring homomorphism 
	$R_{i(\Phi)} \ra R_{\Phi'}$.
	If $a \in \K^\natural$ then $a\oplus \Si a \in i(\K)$ as before and it follows that if $a \in \Phi'$ then
	$a \oplus \Si a \in i(\Phi)$.
	We claim that for any $a \in \Phi'$, the diagram
	\[\xymatrix{
			R_{i(\Phi)} \ar[r]& R_{\Phi'} \\
			R_{a \oplus \Si a} \ar[ur] \ar[u] & R_a \ar[l]_-{\sim} \ar[u]
		}\]
	commutes, where the bottom row is the isomorphism obtained in Proposition~\ref{proposition:oplus_suspensions}; it will follow that the map 
	$R_{i(\Phi)} \ra R_{\Phi'}$ is surjective.
	The commutativity of the top triangle is immediate. On the other hand,
	consider some $[f] \in R_a$, say with $f \in E_{a^{\tens n}}$.
Suppose for starters that $n=1$.
	Then $[f]$ maps to $[f \oplus \Si f]$ in $R_{a \oplus \Si a}$ and so
	the question (in this case) is whether $[f \oplus \Si f] = [f]$ in $R_{\Phi'}$.
	This is readily verified: 
	\begin{multline*}
		 [f \oplus \Si f] = [(f \oplus \Si f) \tens a] = [(f \tens a) \oplus (\Si f \tens a)] = \\ 
		[(a \tens f)\oplus (\Si a \tens f)] = [(a \oplus \Si a) \tens f] = [f].
	\end{multline*}
	For general $n \ge 1$, regarding $(a \oplus \Si a)^{\tens n}$ as a $\oplus$-sum of suspensions of $a \tens a$, 
	an element $f \in E_{a^{\tens n}}$ is sent to a $\oplus$-sum of suspensions of $f$ and it comes down to showing that
	$[f] = [\Si^{i_1} f \oplus \Si^{i_2} f \oplus \cdots \oplus \Si^{i_k} f]$ in $R_{\Phi'}$ (which can
	be verified in a similar manner).

	On the other hand, if $[f] \in R_{i(\Phi)}$ is an element that is sent to zero in $R_{\Phi'}$ 
	then $a \tens f = 0$ for some $a \in \Phi'$ and 
	$(a \oplus \Si a)\tens f \simeq (a \tens f) \oplus \Si (a \tens f) = 0$ shows that $[f] = 0$ in $R_{i(\Phi)}$.
	Therefore the map
	$R_{i(\Phi)} \ra R_{\Phi'}$ is also injective.
	It is clear that under this isomorphism $\rho_{i(\Phi)} = \rho_{\Phi'}$
	while
	Proposition~\ref{proposition:idempotent_invariant} implies that
	$\rho_\Phi = \rho_{i(\Phi)}$ after identifying $\cZ \simeq \cZ'$.
\end{proof}
In other words, we have established that $\K$ and $\K^\natural$ give precisely the same object comparison maps and precisely the same closed set comparison maps.

\section{Topological results}
Throughout this section let $\K$ be a tensor triangulated category and let $\Phi \subset \K$ be a non-empty set of objects closed under the $\tens$-product.
\begin{proposition}
	If $\cZ_\Phi$ is connected then $\Spec(R_\Phi)$ is connected.
\end{proposition}
\begin{proof}
By the results of the last section it suffices
	to prove the result under the additional hypothesis that $\K$ is
	idempotent-complete.  If $\Spec(R_\Phi)$ is disconnected then there is a
	non-trivial idempotent $[e_X]$ in the ring $R_\Phi$. 
By Lemma~\ref{lemma:concrete_equivalence}, $[e_X] = [e_X]^2
	= [e_X^2]$ 
	implies that there is an $a \in \Phi$
such that $e_X \tens a \tens X = X \tens a \tens e_X^2 = (X \tens a \tens e_X)^2$
while $e_X \tens a \tens X = X \tens a \tens e_X$ since $e_X$ is $\tens$-balanced.
	Moreover, $[e_X]\neq
	0$ implies $X \tens a \tens e_X \neq 0$ and $[e_X]\neq 1$ implies 
	$X \tens a \tens
	e_X \neq \id_{X \tens a \tens X}$, so $f := X \tens a \tens e_X$
	is a non-trivial idempotent endomorphism of $X \tens a \tens X$.  Moreover,
	it is contained in $E_{X\tens a \tens X}$ and hence gives an element
	$[f]$ of $R_\Phi$.

	Since $\K$ is idempotent-complete, 
	there is an associated decomposition $X \tens a\tens X 
	\simeq a_1 \oplus a_2$ for two non-zero
	objects $a_1$ and $a_2$ such that $f$
	becomes the matrix~$\big(\begin{smallmatrix} 1 & 0 \\ 0 & 0
	\end{smallmatrix}\big)$.  In particular, $\cone(f) \simeq \Si a_2 \oplus
	a_2$ and $\cone(\id_{X \tens a \tens X} -f)\simeq \Si a_1 \oplus a_1$.
	It follows that $\cZ_\Phi \cap \supp(a_1) \cap \supp(a_2) = \emptyset$
	because otherwise there would be a $\P \in \cZ_\Phi$ with $\cone(f)
	\notin \P$ and $\cone(\id_{X\tens a \tens X} - f) \notin \P$, which
	would imply that both $[f]$ and $1-[f]$ are contained in the prime
	$\rho_\Phi(\P)$. 
	We conclude that $\cZ_\Phi = \cZ_\Phi \cap {\supp(X \tens a \tens X)} = (\cZ_\Phi \cap \supp(a_1)) \sqcup (\cZ_\Phi \cap \supp(a_2))$
	is a disjoint union of closed sets, and it remains to show that $\cZ_\Phi
	\cap \supp(a_i) \neq \emptyset$ for $i=1,2$.

	If $\cZ_\Phi \cap \supp(a_1)$ were empty then the quasi-compactness of
	$\Spc(\K)$ would imply that there is a $c \in \Phi$ such that $\supp(c)
	\cap \supp(a_1) = \emptyset$ and hence that $c \tens a_1$ is
	$\tens$-nilpotent.  But observe that under the identification $X\tens a
	\tens X \simeq a_1 \oplus a_2$, the endomorphism $f^{\tens n}$ becomes a
	matrix with all zero entries except for $\id_{a_1^{\tens n}}$ at one
	position along the diagonal.  From this it is clear that $c^{\tens n} \tens
	f^{\tens n} = 0$ for some $n \ge 1$ since $c \tens a_1$ is
	$\tens$-nilpotent.  But then $[e_X] = [f] = [f^n] = [f^{\tens n}]
	= [c^{\tens n} \tens f^{\tens n}] = 0$ in the ring $R_\Phi$, which contradicts
	the fact that $[e_X]$ is nontrivial.  A similar argument shows that if
	$\cZ_\Phi \cap \supp(a_2)$ were empty then $[e_X]=1$.
\end{proof}
For the converse of the above result, we need to add additional assumptions.
\begin{proposition}
\label{proposition:connected_hard}
Suppose $\K$ is rigid and $\cZ_\Phi$ is Thomason.
	If $\Spec(R_\Phi)$ is connected then $\cZ_\Phi$ is connected.
\end{proposition}
\begin{proof}
	By passing to the idempotent completion
	it suffices to prove the result
	under the additional hypothesis that $\K$ is idempotent-complete (cf.~the
	results 
	in the last section and note that
	$\K$ rigid implies that $\K^\natural$ is also rigid \cite[Proposition~2.15(i)]{Balmer_Supports}).
	Suppose $\cZ_\Phi = Y_1 \sqcup Y_2$ is a disjoint union of non-empty closed sets.
	Each $Y_i$ is quasi-compact (being closed) and it follows from the fact that they are disjoint and that $\cZ_\Phi$ is Thomason
	that each $Y_i$ is Thomason.
It also follows from the definition $\cZ_\Phi := \bigcap_{X \in \Phi}\supp(X)$ and the fact that $\Spc(\K) \setminus \cZ_\Phi$ is quasi-compact that $\cZ_\Phi = \supp(a)$ for some $a\in \Phi$.
Since $\K$ is rigid and idempotent-complete, the generalized Carlson theorem \cite[Remark~2.12]{Balmer_Supports} implies that there exist $a_1, a_2 \in \K$ such
	that $a \simeq a_1 \oplus a_2$ and $\supp(a_i) = Y_i$ for $i=1,2$.
	Since $\K$ is rigid and $\supp(a_1) \cap \supp(a_2) = \emptyset$ it follows 
	\cite[Corollary~2.8]{Balmer_Supports} that $[a_i,a_j] =0$ for $i\neq j$.
	This implies that the idempotent
	$f := \left(\begin{smallmatrix}1&0\\0&0\end{smallmatrix}\right)$
	is $\tens$-balanced ($f \tens a = a \tens f$) and hence provides an idempotent
	element $[f]$ of the ring $R_\Phi$.
	If $[f]$ was a trivial idempotent it would follow that
	$b \tens a_1 = 0$ or $b \tens a_2 = 0$ for some $b \in \Phi$.
	But $Y_i = \supp(a_i) = \supp(b \tens a_i)$ since $\supp(a_i) \subset \cZ_\Phi \subset \supp(b)$ and
	so $b \tens a_i =0$ contradicts the fact that $Y_i$ is non-empty.
	We conclude that $[f]$ is a nontrivial
	idempotent of the commutative ring
	$R_\Phi$ and hence that $\Spec(R_\Phi)$ is
	disconnected.  \end{proof}
\begin{remark}
	Example~\ref{example:posets} below will provide some (non-rigid) examples
	of tensor triangulated categories for which the conclusion 
	of Proposition~\ref{proposition:connected_hard} does not hold.
\end{remark}
\begin{remark}
For any (graded-)commutative graded ring $R^\astt$, 
the space $\Spech(R^\astt)$ is connected if and only if $\Spec(R^0)$ is connected, so the graded version of the above statements also holds.
\end{remark}
\begin{lemma}\label{lemma:thomason_killer}
	If $c$ is an object in $\K$ with $\cZ_\Phi \subset \supp(c)$ then there is some $a \in \Phi$ such that $\supp(a) \subset \supp(c)$.
\end{lemma}
\begin{proof}
	Since $\Spc(\K) \setminus \supp(c)$ is quasi-compact (recall Lemma~\ref{lemma:thomason_subsets}), it follows from
	the definition $\cZ_\Phi := \bigcap_{X\in\Phi} \supp(X)$ that there is a finite collection of objects $X_1, \ldots, X_n \in \Phi$
	such that $\supp(X_1) \cap \cdots \cap \supp(X_n) \subset \supp(c)$ and we can take $a := X_1 \tens \cdots \tens X_n \in \Phi$.
\end{proof}
\begin{proposition}
	\label{proposition:proper_pullbacks}
	Consider the map $\rho_\Phi^\astt : \cZ_\Phi \ra \Spech(R_\Phi^\astt)$ and
	any
	closed set $\cW \subset \Spech(R_\Phi^\astt)$.
	If $(\rho_\Phi^\astt)^{-1}(\cW) = \cZ_\Phi$ then $\cW = \Spech(R_\Phi^\astt)$.
	In other words, the preimage of a proper closed set remains proper.
\end{proposition}
\begin{proof}
	Let $\mathcal{W} = V(I)$ for some homogeneous ideal $I \subset R_\Phi^\astt$ 
	and let $[f]$ be an arbitrary homogeneous element of $I$.
	It follows from our hypothesis that  
	 $[f] \in \rho_\Phi^\astt(\P)$ for every
	 $\P \in \cZ_\Phi$ and hence that $\cZ_\Phi \subset \supp(\cone(f))$.
	 Lemma~\ref{lemma:thomason_killer} then implies that there is some $a \in \Phi$
	 such that $\supp(a) \subset \supp(\cone(f))$.
	 It follows that $a^{\tens i} \in \langle \cone(f) \rangle$ for some $i \ge 1$ and hence 
	 that $a^{\tens i} \tens f^{\tens j} = 0$ for some $i,j \ge 1$ by Lemma~\ref{lemma:Nf}.
	But then $[f]^j= [f^{\tens j}] = [a^{\tens i}\tens f^{\tens j}] = 0$ so that $[f]$ is a nilpotent element of $R_\Phi$.
	Since every homogeneous element of $I$ is nilpotent, 
	$\cW = V(I) = \Spech(R_\Phi^\astt)$.
\end{proof}
\begin{remark}
	\label{remark:ungraded_version_of_preimages}
	One easily verifies that the ungraded
	maps $\rho_\Phi$ also have the
	property that preimages of proper closed subsets
	remain proper, either by carrying out an
	ungraded version of the above proof, or as a corollary of the above proposition by
	observing that the 
	map $(-)^0 : \Spech(R_\Phi^\astt) \ra \Spec(R_\Phi)$ has this property.
\end{remark}
\begin{corollary}\label{corollary:dense_image}
	The maps $\rho_\Phi^\astt : \cZ_\Phi \ra \Spec(R_\Phi^\astt)$ and
	$\rho_\Phi : \cZ \ra \Spec(R_\Phi)$ have dense images.
	Consequently, if $\cZ_\Phi$ is irreducible then $\Spec(R_\Phi^\astt)$ and $\Spec(R_\Phi)$ are also irreducible.
\end{corollary}
\begin{lemma}\label{lemma:rho_covers}
If $\K$ is rigid and $[f]$ is a homogeneous non-unit in the ring $R_\Phi^\astt$
then there is 
	some $\P \in \cZ_\Phi$ such that $[f] \in
	\rho_\Phi^\astt(\P)$.
\end{lemma}
\begin{proof}
	If there were no such $\P$ then $\cZ_\Phi \cap \supp(\cone(f)) = \emptyset$ which
	implies by the quasi-compactness
	of $\Spc(\K)$ that
	$\supp(a) \cap \supp(\cone(f)) = \emptyset$
	for some $a \in \Phi$. 
	But $\supp(a \tens \cone(f)) = \emptyset$ implies by the rigidity of $\K$ that $a \tens \cone(f)= 0$.
	Thus $a \tens f$ is invertible and hence $[f]$ is a unit 
	in $R_\Phi^\astt$.
\end{proof}
\begin{proposition}
Suppose $\K$ is rigid.
	If the map $\rho_\Phi : \cZ_\Phi \ra \Spec(R_\Phi)$ is a constant map then $\Spec(R_\Phi)$ is a point.
	Similarly, if the map $\rho_\Phi^\astt : \cZ_\Phi \ra \Spech(R_\Phi^\astt)$ is a constant map then
	$\Spech(R_\Phi^\astt)$ is a point.
\end{proposition}
\begin{proof}
	Suppose $\p$ is a prime in $R_\Phi$ such that $\rho_\Phi(\P) = \p$ for all
	$\P \in \cZ_\Phi$. Lemma~\ref{lemma:rho_covers} then implies that $\p$
	contains every non-unit of $R_\Phi$.
	It follows that $R_\Phi$ is local with $\p$ its unique maximal ideal. 
On the other hand, Corollary~\ref{corollary:dense_image} implies that $\{\p\}$ is dense, and so
$\{\p\} =
	\closure{\{\p\}} = \Spec(R_\Phi)$.
	An identical argument works in the graded case.
\end{proof}
\begin{remark}
	The theorems proved in this section are new even in the case of Balmer's original map $\rho_\unit$.
	For example, Corollary~\ref{corollary:dense_image} implies that $\rho_\unit$ always has dense image without any assumptions on the
	category $\K$.
	This result is particularly interesting in light of the surjectivity criteria for $\rho_\unit$ established in 
	\cite{Balmer_SSS}.
\end{remark}
\begin{remark}
	One strategy for studying the spectrum is to iteratively build a filtration of closed subsets by pulling back filtrations via our closed set comparison maps.
	More precisely, we begin with the trivial filtration $\{\Spc(\K)\}$ and at each iterative step we consider every closed set $\cZ$ in the filtration
	(or only those that were newly added at the last step) and refine the filtration below $\cZ$ by 
	pulling back the filtration of closed subsets of $\Spech(R_\cZ^\astt)$ via the map
	$\rho_\cZ^\astt : \cZ \ra \Spech(R_\cZ^\astt)$.
A key result for this idea is Proposition~\ref{proposition:proper_pullbacks}
which asserts that proper closed subsets of $\Spech(R_\cZ^\astt)$ pull back to proper
subsets of $\cZ$. This implies that the process continues to refine the spectrum for as long as the spaces $\Spech(R_\cZ^\astt)$ are non-trivial.
However, an obstacle is the possibility that 
$\Spech(R_\cZ^\astt)$ might be trivial for some non-trivial closed set $\cZ$, in which case the internal structure of $\cZ$ would remain hidden.
In fact, there are examples of tensor triangulated categories for which it seems the process may hit the wall at the very first step (cf.~Example~\ref{example:posets} below).
Nevertheless, these examples are non-rigid and there is some hope that under suitable hypotheses 
this obstacle might disappear.
\end{remark}
\begin{example}\label{example:posets}
Let $P$ be a finite poset, $k$ a field, and let $\K := D^b(\rep_k(P))$ be the
derived category of finite-dimensional $k$-linear representations of $P$.  The
abelian category $\rep_k(P)$ has an exact vertex-wise $\tens$-structure and
$\K$ inherits the structure of a tensor triangulated category.  Recognizing
that representations of $P$ are the same thing as quiver representations of the
associated Hasse diagram with full commutativity relations, one sees that the
work of \cite{sierra_preprint} completely describes the spectrum $\Spc(\K)$.
It turns out to be rather trivial: a discrete space with points corresponding
to the elements of $P$.  For a representation $V$ regarded as a complex
concentrated in degree $0$, $\supp(V) = \{x \in P \mid V_x \neq 0\}$.  If $P$
has a least element then $\unit$ is projective (so $[\unit,\unit]_i = 0$ for $i
\neq 0$) and $[\unit,\unit]\simeq k$ by inspection.  There are thus many
examples where $\Spech(R_\unit^\astt)$ is trivial (in particular, connected)
but $\Spc(\K)$ is disconnected.  This doesn't contradict
Proposition~\ref{proposition:connected_hard} because these examples of derived
quiver representations are not rigid.  For example, consider the simplest
non-trivial example: $P = (1 \ra 2)$.  Let $S_1 = (k\ra 0)$ and $S_2 = (0\ra
k)$ be the two simple representations.  There is an obvious exact triangle $S_2
\ra \unit \ra S_1 \ra \Si S_2$.  If~$\K$ were rigid then the fact that
$\supp(S_1) \cap \supp(\Si S_2) = \supp(S_1) \cap \supp(S_2) = \emptyset$
implies that $[S_1,\Si S_2]=0$, and hence that the exact triangle splits:
$\unit \simeq S_1 \oplus S_2$ which is evidently not the case.
\end{example}

\section{Triangular localization}
An important feature of the original unit comparison maps developed in
\cite{Balmer_SSS} is that it is possible to localize the category $\K$ with
respect to primes in the ring $[\unit,\unit]$, which enables one to reduce
questions about the unit comparison maps to the case where the target ring is
local.  Fortunately, one may establish such a localization technique for our
more general comparison maps.
\begin{theorem}
	\label{theorem:graded_localization}
	Let $\K$ be a tensor triangulated category,
	$\Phi \subset \K$ a non-empty set of objects closed under the $\tens$-product,
	$S \subset R_{\K,\Phi}^\astt$ a multiplicative set of even (hence central) homogeneous elements, and 
	$q:\K \ra \K/\J$ the Verdier quotient of $\K$ with
	respect to the thick $\tens$-ideal $\J := \langle\cone(s)
		\mid [s] \in S \rangle$.
	Then $R_{\K/\J,\Phi}^\astt$ is isomorphic to the ring-theoretic localization $S^{-1}(R_{\K,\Phi}^\astt)$ 
	and we have a diagram
\begin{equation}\label{diagram:graded_localization}
	\begin{gathered}
		\xymatrix @C=1em {	
\cZ_{\K/\J,\Phi} \ar[d]^{\rho^\astt_{\K/\J,\Phi}} \ar@{^{(}->}[rrr] &&& \cZ_{\K,\Phi} \ar[d]^{\rho^\astt_{\K,\Phi}}\\
			\Spech(R_{\K/\J,\Phi}^\astt) \ar@{}[r]|-*[@]{=} &\Spech(S^{-1}R_{\K,\Phi}^\astt) \ar@{^{(}->}[rr] && \Spech(R_{\K,\Phi}^\astt).
		}
	\end{gathered}
	\end{equation}
	that is commutative and cartesian: $\cZ_{\K/\J,\Phi} \cong \{ \P \in \cZ_{\K,\Phi} \mid \rho_{\K,\Phi}^\astt(\P) \cap S = \emptyset \}$.
\end{theorem}
\begin{remark}
If $S \subset R^0$ then
$(S^{-1}R^\astt)^0 = S^{-1}R^0$ and one
readily verifies that 
applying $(-)^0$ to the bottom row of
\eqref{diagram:graded_localization}
yields a commutative, cartesian diagram
\begin{equation}\label{diagram:ungraded_localization}
		\begin{gathered}
		\xymatrix @C=1em {	
\cZ_{\K/\J,\Phi} \ar[d]^{\rho_{\K/\J,\Phi}} \ar@{^{(}->}[rrr] &&& \cZ_{\K,\Phi} \ar[d]^{\rho_{\K,\Phi}}\\
			\Spech(R_{\K/\J,\Phi}) \ar@{}[r]|-*[@]{=} &\Spech(S^{-1}R_{\K,\Phi}) \ar@{^{(}->}[rr] && \Spech(R_{\K,\Phi}).
		}
	\end{gathered}
	\end{equation}
This gives the ungraded version of our localization result.
\end{remark}
The remainder of this section is devoted to proving the theorem.
For purposes of clarity we will prove the ungraded version---the
graded result stated in 
Theorem~\ref{theorem:graded_localization} is proved in the
same way but the
notation gets more cumbersome.
Thus, for the rest of the section we fix a multiplicative set $S \subset R_{\K,\Phi}$.
For an element $[s] \in S$, we'll use the notation $X_s$ to indicate that the representative $s$ is an endomorphism of $X_s$.
For morphisms in $\K/\J$ we'll use the left fractions 
of
\cite[Section~3]{Krause_Localization}.
It is immediate from the definition of $\J$ that the canonical map
$R_{\K,\Phi} \ra R_{\K/J,q(\Phi)}$ factors as
\[ \xymatrix@1{ R_{\K,\Phi} \ar[r]^-{\eps} & S^{-1}(R_{\K,\Phi}) \ar[r]^-{i} & R_{\K/\J,q(\Phi)} }\]
where $\eps$ is the canonical localization map.
In order to show that $i$ is an isomorphism we will need the following three lemmas:
\begin{lemma}
	\label{lemma:localization_nilpotence}
	If $a \in \J$ then there is a representative $s$ of an element $[s] \in S$ with
	the property that $a \tens s = 0$.
\end{lemma}
\begin{proof}
	One verifies using standard techniques that the collection of objects $a \in \K$ for which there is a representative $s$ of an element $[s]\in S$ such that $a \tens s^{\tens n} = 0$ for some $n\ge 1$ forms a thick $\tens$-ideal of $\K$.
	It contains each $\cone(s)$ by Lemma~\ref{lemma:killscone} and hence it contains $\J$.
	Since $S$ is multiplicative, $[s^{\tens n}] = [s]^n \in S$.
\end{proof}
\begin{lemma}
	\label{lemma:localization_trick}
	If $f:a \ra b$ and $g: c\ra d$ are morphisms such that
	$\cone(f) \tens g = 0$
	then there exists a morphism
	$u: b \tens c \ra a\tens d$ such that $a \tens g = u\circ(f\tens c)$
	and a morphism $v : b \tens c \ra a \tens d$ such that
	$b \tens g = (f \tens d)\circ v$.
\end{lemma}
\begin{proof}
	The morphisms $u$ and $v$ are obtained from the morphism of exact triangles
\[\xymatrix{
		a\tens c \ar[rr]^{f\tens c} \ar[d]_{a\tens g} && b \tens c \ar@{.>}@<-0.15em>[dll]_{u}\ar@{.>}@<0.15em>[dll]^{v} \ar[d]^{b \tens g} \ar[r] & \cone(f) \tens c \ar[d]^{\cone(f)\tens g=0} \ar[r] & \Si(a \tens c) \ar[d]^{\Si(a \tens g)}\\
		a\tens d \ar[rr]_{f\tens d} && b \tens d \ar[r] & \cone(f)\tens d \ar[r] & \Si(a \tens d)
	}\]
by the weak kernel and cokernel properties of exact triangles.
\end{proof}
\begin{lemma}
	\label{lemma:twooutofthreeinvertible}
	Let $f : a \ra a$ be an endomorphism in a tensor triangulated category.
	If there exists a $\tens$-balanced automorphism $\sigma : a \ra a$ such that
	$f \circ \sigma$ is $\tens$-balanced then $f$ is $\tens$-balanced.
\end{lemma}
\begin{proof}
	This is easily verified from the definitions recalling that $E_a$ is an inverse closed subring of $[a,a]$.
\end{proof}
\begin{proposition}
	The map $i: S^{-1}(R_{\K,\Phi}) \ra R_{\K/J,q(\Phi)}$ is injective.
\end{proposition}
\begin{proof}
	Consider an element $[f]/[s] \in S^{-1}(R_{\K,\Phi})$. If $i([f]/[s]) = 0$ then
	\begin{equation*}
		[\lfraction{X_f\tens X_s}{f\tens 1}{X_f \tens X_s}{1\tens s}{X_f \tens X_s}]=0
	\end{equation*}
	in $R_{\K/\J,q(\Phi)}$ which implies that there is an $a \in \Phi$ such that
	\[\lfractionstretch{a\tens X_f\tens X_s}{1\tens f\tens 1}{a \tens X_f \tens X_s}{1\tens s\tens 1}{a\tens X_f \tens X_s}{4em}=0\]
	as a morphism in $\K/J$.
It follows that there is a morphisms $k:a \tens X_f \tens X_s \ra b$ in~$\K$ with $\cone(k) \in \J$ such that
$k \circ (a \tens f \tens X_s) = 0$.
By Lemma~\ref{lemma:localization_nilpotence} there is some $[t] \in S$
such that
$\cone(k) \tens t  = 0$ and hence by Lemma~\ref{lemma:localization_trick}
there is a morphism
$u:b \tens X_t \ra a\tens X_f \tens X_s \tens X_t$ such that 
$u \circ (k\tens X_t) = a\tens X_f \tens X_s \tens t$.
In the ring $R_{\K,\Phi}$ we then have
\begin{align*}
	[t]\cdot [f] = [(a\tens X_f \tens X_s \tens t)\circ (a\tens f\tens X_s \tens X_t)] = 0
\end{align*}
and we conclude that $[f]/[s]=([t]\cdot [f])/([t]\cdot [s]) =0$ in $S^{-1}(R_{\K,\Phi})$.
\end{proof}
\begin{proposition}
	\label{proposition:localization_surjective}
	The map $i:S^{-1}(R_{\K,\Phi}) \ra R_{\K/\J,\Phi}$ is surjective.
\end{proposition}
\begin{proof}
Consider an element $[\lfraction{a}{f}{b}{\sigma}{a}]\in
R_{\K/\J,q(\Phi)}$. By Lemma~\ref{lemma:localization_nilpotence} there is some
$[s] \in S$ such that $\cone(\sigma) \tens s=0$.
It then follows from Lemma~\ref{lemma:localization_trick} 
that there exists a morphism $u:b \tens X_s \ra a\tens X_s$
such that $u\circ(\sigma \tens X_s)=a\tens s$.
We thus have an equality of left fractions
\begin{equation*}
	\lfractionstretch{a\tens X_s}{f\tens 1}{b\tens X_s}{\sigma \tens 1}{a \tens X_s}{3em} 
	= \lfractionstretch{a\tens X_s}{u\circ(f\tens 1)}{a\tens X_s}{1\tens s}{a \tens X_s}{3em}.
\end{equation*}
The left-hand side is an element of $E_{\K/\J,a \tens X_s}$ and so it follows from
Lemma~\ref{lemma:twooutofthreeinvertible} that
\begin{equation*}
	\lfractionstretch{a\tens X_s}{u \circ(f \tens 1)}{a\tens X_s}{\id}{a\tens X_s}{4em}
\end{equation*}
is an element of $E_{\K/\J,X\tens X_s}$.
The claim then follows from Lemma~\ref{lemma:surjectivity_lemma} below.
\end{proof}
\begin{lemma}
	\label{lemma:surjectivity_lemma}
	If $f:a \ra a$ is an endomorphism of an object $a \in \Phi$ such that 
	$q(f)$ is $\tens$-balanced as an endomorphism in $\K/\J$ then
	$[q(f)]\in R_{\K,q(\Phi)}$ is contained in the image of $i:S^{-1}(R_{\K,\Phi}) \ra R_{\K/\J,\Phi}$.
\end{lemma}
\begin{proof}
	It follows from the equality of fractions
	\[ \lfraction{a\tens a}{a \tens f}{a\tens a}{\id}{a\tens a}
		= \lfraction{a\tens a}{f \tens a}{a \tens a}{\id}{a\tens a} \]
	that there is a morphism $\tau : a \tens a \ra b$ in $\K$ such that
	$\cone(\tau)\in\J$ and $\tau \circ (a \tens f) = \tau \circ (f \tens a)$.
	By Lemma~\ref{lemma:localization_nilpotence} and Lemma~\ref{lemma:localization_trick}
	there is a $[t] \in S$ such that $\cone(\tau) \tens t = 0$
	and a morphism 
	$u : b \tens X_t \ra a \tens a \tens X_t$
	such that $u \circ (\tau \tens X_t) = a \tens a \tens t$.
	It follows that $a \tens f \tens t = f \tens a \tens t$ and we conclude 
	using Lemma~\ref{lemma:twisting_lemma} that $f \tens t$ is an element of $E_{a\tens X_t}$.
	Thus 
	\[[\lfraction{a}{f}{a}{\id}{a}]\cdot[\lfraction{X_t}{t}{X_t}{\id}{X_t}]=
		[\lfraction{a\tens X_t}{f\tens t}{a \tens X_t}{\id}{a \tens X_t}]\]
	is contained in the image of $i$.
\end{proof}

\section{Stable homotopy theory}\label{section:stable_homotopy_theory}
Let $\SH$ denote the stable homotopy category of spectra and let $\SHf$ denote
its full subcategory of finite spectra.  For a fixed prime~$p$, one usually
defines the stable homotopy category of finite $p$-local spectra $\SHfp$ as the
full subcategory of $\SH$ consisting of all spectra isomorphic to the
$p$-localization of a finite spectrum, but for our purposes it is convenient to
recognize that this is equivalent to the Verdier quotient of $\SHf$ by the
finite $p$-acyclic spectra.

The purpose of this section is to illustrate the iterative method for examining
fibers of comparison maps in the example of the stable homotopy category of
finite spectra $\SHf$.  This will depend on a description of the graded centers
of endomorphism rings of finite $p$-local spectra provided by \cite{HS98} which
affords a description of the ring 
\[A_X^\astt := \Center([X,X]_\astt)\cap	E_X^\astt\] 
for every finite $p$-local spectrum~$X$, but not of the non-commutative ring
$E_X^\astt$, nor the ring $R_X^\astt = \colim_{n \ge 1} E_{X^{\tens n}}^\astt$.
For this reason, we'll have to settle for the unnatural comparison maps of
Section~\ref{section:basic_construction}. 

First, let us briefly recall the classification of thick subcategories in
$\SHfp$.  For each $n \ge 1$, the $n$th Morava $K$-theory spectrum $K(n)$ is a
$p$-local ring spectrum whose coefficient ring $K(n)_\astt := \pi_\astt(K(n))$
is the graded-field $\Fp[\v_n,\v_n^{-1}]$ with $|\v_n| = {2(p^n-1)}$.  The
collection of Morava $K$-theories is completed by $K(0):=H\Q$ and
$K(\infty):=H\Fp$ and for each $0 \le n \le \infty$ provides a stable
homological functor
\[ K(n)_\astt(-) : \SHpf \ra K(n)_\astt\text{-grMod}\]
which is in fact a strong $\tens$-functor. 
It follows, using the fact that $K(n)_\astt$ is a graded-field, that the kernel
of this functor is more than a thick subcategory---it is a prime $\tens$-ideal;
i.e.\ a point of $\Spc(\SHpf)$.
Conforming to the notation of \cite{HS98} we define
$\C_0 := \SHpf$, $\C_n := \ker(K(n-1)_\astt(-))$ for $n\ge 1$, 
and $\C_\infty := \ker(K(\infty)_\astt(-))$.
These categories fit into a filtration
\[0 = \C_\infty \subsetneq \cdots \subsetneq \C_{n+1} \subsetneq \C_n \subsetneq \cdots \subsetneq \C_1 \subsetneq \C_0 = \SHfp\]
and the Hopkins-Smith classification theorem 
asserts that \emph{every} thick $\tensor$-ideal is of the form $\C_n$ for
some $0 \le n \le \infty$.
From a different point of view, this shows that $\Spc(\SHfp)$ consists of a sequence of points
\[\C_1 \ra \C_2 \ra \C_3 \ra \cdots \ra \C_n \ra \C_{n+1} \ra \cdots \ra \C_\infty = (0) \]
where $\ra$ indicates specialization: $\overline{\{\C_n\}} = \{ \C_i \mid i \ge n\}$.

The main results from \cite{HS98} which allow for a description of
our higher comparison maps arise from their study of non-nilpotent
(graded) endomorphisms of finite $p$-local spectra.  Recall that
one statement of the Nilpotence Theorem is that an
endomorphism $f: \Si^d X \ra X$ of a finite $p$-local spectrum 
is nilpotent iff $K(n)_\astt(f)$ is nilpotent for all $0 \le n < \infty$.
This motivates the following definition
which aims to pin down those
non-nilpotent endomorphisms that are as
simple as possible.
\begin{definition}\cite[Definition~8]{HS98}
	Let $n \ge 1$ and let $X$ be a finite $p$-local spectrum.
	An endomorphism $f:\Si^d X \ra X$ is a \emph{$\v_n$-selfmap} if
	$K(n)_\astt(f)$ is an isomorphism, and $K(i)_\astt(f)$ is nilpotent for $i\neq n$.
\end{definition}
It follows from the definitions (and the Nilpotence Theorem) that
	if $X \in \C_{n+1}$ then $\v_n$-selfmaps are the same thing as nilpotent selfmaps
	but 
	that if $X \notin \C_{n+1}$ then $\v_n$-selfmaps are never
	nilpotent.
	It is also easily shown that $X \in \C_n$ is a necessary
	condition for the existence of a $\v_n$-selfmap.
Thus, the notion is mostly of interest for $X\in \C_n \setminus \C_{n+1}$.	
The first result of substance is that $\v_n$-selfmaps exist:
\begin{theorem}\label{theorem:vn_existence}\cite[Theorem~9]{HS98}
	A finite $p$-local spectrum $X$ admits a $\v_n$-selfmap if and only if $X\in \C_n$.
\end{theorem}
The most important properties about $v_n$-selfmaps, other than the fact that they exist, are that they are asymptotically unique and asymptotically central:
\begin{proposition}\label{proposition:vn_are_unique}\cite[Corollary~3.7]{HS98}
	\label{uniqueness} If $f$ and $g$ are two $v_n$-selfmaps of $X$ then
		$f^i = g^j$ for some $i,j \ge 1$.
\end{proposition}
\begin{proposition}\cite[Corollary~3.8]{HS98}
	If $f$ is a $v_n$-selfmap of $X$ then some power of $f$ is contained in the center of $[X,X]_\astt$.
\end{proposition}
For our purposes, we need to include the following result:
\begin{lemma}\label{lemma:vn_are_asymptotically_tensor_balanced}
	If $f$ is a $v_n$-selfmap of $X$ then some power of $f$ is $\tens$-balanced.
\end{lemma}
\begin{proof}
	This is straightforward in light of Proposition~\ref{proposition:vn_are_unique} by recognizing
	that $f \tens X$ and $X \tens f$ are two $v_n$-selfmaps of $X \tens X$.
\end{proof}
The notion of a $v_n$-selfmap leads to a very complete description of the centers of graded endomorphism rings of finite $p$-local spectra---up to nilpotents.
\begin{theorem}\cite[Corollary~5.5, Proposition~5.6]{HS98}
	\label{theorem:class}
	Let $X$ be a finite \mbox{$p$-local} spectrum and let $f$ be a graded endomorphism of $X$ which is in the center of $[X,X]_\astt$.
	If $f$ is degree 0 then some power of $f$ is a multiple of the identity; otherwise,
	$f$ is nilpotent or a $v_n$-selfmap.
\end{theorem}
\begin{corollary}\label{corollary:two_point_description}
	Let $n \ge 1$ and $X \in \C_n \setminus \C_{n+1}$.
	The space $\Spech(A_X^\astt)$ consists of two points:
	a generic point consisting of the homogeneous nilpotents and a closed point
	consisting of the homogeneous non-units. The closed point is of the form $\sqrt{(f)}$ for
	any $\tens$-balanced $v_n$-selfmap $f:\Si^d X \ra X$.
\end{corollary}
\begin{proof}
	If $f \in A_X^0$ then $f^k = m.\id_X$ for some $k \ge 1$ and $m \in \Z$.
	If $p \mid m$ then it follows from the Nilpotence Theorem and the fact that $X \in \C_1$ that $f$ is nilpotent, 
	while if $p \nmid m$ then $m.\id_X$ is an isomorphism.
	Thus, every element of degree zero in $A_X^\astt$ is either nilpotent or a unit.
	On the other hand, every element of non-zero degree is either nilpotent or a $v_n$-selfmap;
	moreover, $v_n$-selfmaps are not nilpotent since $X \notin \C_{n+1}$, nor are they units.
	It follows that the homogeneous non-units form an ideal $\mathfrak{m}$ which is necessarily the unique
	maximal homogeneous ideal of $A_X^\astt$. On the other hand, the 
	ideal of homogeneous nilpotents $\mathfrak{n}$ is readily seen to be prime just using the fact that the product of two $v_n$-selfmaps is again a $v_n$-selfmap and the product of a unit and a $v_n$-selfmap is again a $v_n$-selfmap.
	Since a $v_n$-selfmap exists in $A_X^\astt$ (by Lemma~\ref{lemma:vn_are_asymptotically_tensor_balanced} and Theorem~\ref{theorem:vn_existence}), $\mathfrak{n} \subsetneq \mathfrak{m}$.
	Moreover, if $\mathfrak{p}$ is a homogeneous prime ideal then $\mathfrak{n} \subsetneq \mathfrak{p}$ implies that there is a $v_n$-selfmap $f$ contained in $\mathfrak{p}$.
	By the asymptotic uniqueness of $v_n$-selfmaps, it follows that every $v_n$-selfmap is contained in $\mathfrak{p}$, so that $\mathfrak{p} = \mathfrak{m}$.
	We conclude that the only homogeneous primes of $A_X^\astt$ are $\mathfrak{n}$ and~$\mathfrak{m}$.
	Finally note that the asymptotic uniqueness of $v_n$-selfmaps implies that $\mathfrak{m} = \sqrt{(f)}$ for any $\tens$-balanced $v_n$-selfmap $f : \Si^d X \ra X$.
\end{proof}
\begin{lemma}\label{lemma:cone_next}
	Let $n \ge 1$ and $X \in \C_n \setminus \C_{n+1}$. If $f:\Si^d X \ra X$ is a $v_n$-selfmap
	then
	$\cone(f) \in \C_{n+1} \setminus \C_{n+2}$.
\end{lemma}
\begin{proof}
	This is a routine exercise using the long exact sequences obtained by applying Morava $K$-theories to an exact triangle for $f:\Si^d X \ra X$.
\end{proof}
We are now in a position to examine the structure of $\SHf$ via higher comparison maps.
The starting point is the comparison map for the unit object: the sphere spectrum.
It is well-known that the endomorphism ring of the sphere spectrum is
$\End_{\SHf}(\unit)\cong \Z$.  On the other hand,
$\pi_i(\unit) = 0$ for $i<0$ and all the graded endomorphisms of positive
degree are nilpotent by Nishida's theorem.
It follows that $\Spech(\End_{\SHf}^\astt(\unit)) \simeq \Spec(\End_{\SHf}(\unit))$ and 
that
the graded and ungraded comparison maps coincide.
Moreover, triangular localization with respect to $\mathfrak{p}^\astt \subset \End_{\SHf}^\astt(\unit)$
is the same as triangular localization with respect to $\mathfrak{p}^0 \subset \End_{\SHf}(\unit)$.
\begin{proposition}
The triangular localization of $\SHf$ at a prime $(p) \subset
\End_{\SHf}(\unit)$ is equivalent as a tensor triangulated category to 
the stable homotopy category of finite $p$-local spectra $\SHfp$
while the triangular localization of $\SHf$ at the prime $(0)$ is
equivalent to
the quotient of $\SHf$ by the finite torsion spectra
$\SHftor$.
\end{proposition}
\begin{proof}
	Proving the first statement amounts to showing that the 
thick $\tens$-ideal $\langle \cone(d:\unit \ra \unit) \mid p \nmid d\rangle$
is precisely the set of finite $p$-acyclic spectra.
One inclusion is easily obtained by applying $\pi_\astt(-)\tens \Z_{(p)}$ to an exact triangle for ${d:\unit \ra \unit}$.
On the other hand, if $X$ is a finite $p$-acyclic spectrum then $\pi_i(X)$ is finite with no $p$-torsion for all $i \in \Z$.
For any finite spectrum $X$ it is straightforward to check that $\I_X := \{Y \in \SHf \mid [Y,X]_i \text{ is finite with no $p$-torsion for all } i\in \Z\}$ 
is a thick subcategory of $\SHf$. If $X$ is finite $p$-acyclic then $\I_X$ contains $\unit$ and hence contains the whole of $\SHf$.
In particular $\I_X$ contains $X$ itself and we conclude that $\id_X$ has finite order~$d$ prime to $p$.
Then $\Si X \oplus X \simeq \cone(d.\id_X) \simeq X \tens \cone(d.\id_\unit)$ establishes that $X$ is contained in
$\langle \cone(d:\unit \ra \unit) \mid p \nmid d \rangle$.
A similar approach can be used to prove that $\SHftor$ is equal to $\langle \cone(d:\unit\ra\unit) \mid d \neq 0\rangle$.
\end{proof}
Then consider the map $\rho_\unit : \Spc(\SHf) \ra \Spec(\Z)$.
Triangular localization
with respect to 
the generic point $(0) \in \Spec(\Z)$ gives a map
\[\Spc(\SHf/\SHftor) \ra \Spec(\Q)\]
and we conclude that the fiber over $(0)$
is $V(\SHftor) \subset \SHf$. 
In fact, one can show that $\SHf/\SHftor \simeq D^b(\Q)$ (see \cite[page~113]{Margolis_book}) and hence
the spectrum of $\SHf/\SHftor$ is a single point.  Moreover, $\SHftor$
itself is prime and so we conclude that
the fiber over $(0)$ is the single point $\{\SHftor\}$. 

Next consider the fiber over a closed point $(p) \in \Spec(\Z)$.
Triangular localization provides a map
$\Spc(\SHfp) \ra \Spec(\Z_{(p)})$
and the fiber over the unique closed point $(p) \in \Spec(\Z_{(p)})$ is
$\supp(\cone(p.\id_{\unit})) = \closure{\{\C_2\}}$.
In other words, the fiber includes everything with the exception of a single point: $\C_1$.
The next step is to define $X_1 := \cone(p.\id_\unit)$ and consider
\[ \rho_{X_1,A_{X_1}^\astt}^\astt : \closure{\{ \C_2 \}} = \supp(X_1) \rightarrow \Spech(A_{X_1}^\astt). \]
By Corollary~\ref{corollary:two_point_description} the unique closed point 
of $\Spech(A_{X_1}^\astt)$ is of the form
$\sqrt{(f)}$ for any  
$\tens$-balanced $v_n$-selfmap $f$ of $X_1$ and by Lemma~\ref{lemma:cone_next}
the fiber over this point is
$\supp(\cone(f)) = \closure{\{\C_3\}}$.
Again the fiber consists of everything except for one point
and 
the process continues. At the $n$th step we have an object $X_n$ and a map
\[ \rho_{X_n,A_{X_n}^\astt}^\astt : \closure{\{ \C_{n+1} \}} = \supp(X_n) \ra \Spech(A_{X_n}^\astt). \]
The unique closed point is generated as a radical ideal by any $\tens$-balanced \mbox{$v_n$-selfmap} $f_n$ 
and the fiber over this point is $\supp(\cone(f_n)) = \closure{\{\C_{n+2}\}}$.
Altogether this gives a filtration  of the fiber
\begin{equation}\label{diagram:fiber}
\rho_{\SHfp,\unit}^{-1}((p)) = \closure{\{\C_2\}} \supset
\closure{\{\C_3\}} \supset
\closure{\{\C_4\}} \supset \cdots 
\end{equation}
where exactly one point is removed at each step.
All of this may be better appreciated by considering the picture of $\Spc(\SHf)$ displayed on page~\pageref{diagram:stable_rho_map}.
Triangular localization at $p$ focuses on a single branch and then each successive comparison map chops off the root heading towards $\C_{p,\infty}$.
Note that we obtain every irreducible closed subset of the fiber
\eqref{diagram:fiber} 
except for
the closed point $\closure{\{\C_\infty\}} = \{\C_\infty\}$.
The fact that this point is missed shouldn't be alarming since it corresponds
to an irreducible closed subset which is not Thomason.
If all the rings
involved are noetherian then we can only expect to obtain Thomason closed
subsets since 
our comparison maps
are spectral---when using these strategies we should take arbitrary
intersections of all of the closed subsets that we obtain.
Keep in mind that the Thomason closed subsets are a basis
of closed sets, so if we can obtain all the Thomason closed subsets of the
spectrum then we
have obtained the entire spectrum.

\bibliographystyle{alpha}
\bibliography{biblio}{}
\end{document}